\documentclass[10pt, oneside]{amsart}
\usepackage[utf8]{inputenc}
\usepackage{float}
\usepackage{head}
\usepackage{atozfont}
\usepackage{notation}
\usepackage{ulem}
\usepackage[final]{pdfpages}
\begin{document}
\title{The $C_2$-effective spectral sequence for $C_2$-equivariant connective real $K$-theory}
\author[Kong]{Hana Jia Kong}
\address{Department of Mathematics, The University of Chicago, Chicago, IL 60637}
 \email{hanajk@math.uchicago.edu}
%\date{\today}

\maketitle
%\markboth{Hana Jia Kong}{The $C_2$-effective spectral sequence for $C_2$-equivariant connective $K$-theory}

\begin{abstract}
We construct a $C_2$-equivariant spectral sequence for RO$(C_2)$-graded homotopy groups. The construction is by using the motivic effective slice filtration and the $C_2$-equivariant Betti realization. We apply the spectral sequence to compute the RO$(C_2)$-graded homotopy groups of the completed 
$C_2$-equivariant connective real $K$-theory spectrum. The computation reproves the $C_2$-equivariant Adams spectral sequence results by Guillou, Hill, Isaksen and Ravenel.
\end{abstract}

\tableofcontents

\section{Introduction}

In the $G$-equivariant stable homotopy category, the RO($G$)-graded homotopy groups are fundamental invariants. In this paper, we construct a new $C_2$-equivariant spectral sequence, namely the $C_2$-effective spectral sequence, to compute these groups. We then carry out the computation for a specific example,
% of the $C_2$-equivariant spectrum $\kokq$. 
 which is the first step towards further applications of the $C_2$-effective spectral sequence. 

The $C_2$-effective spectral sequence is constructed from a motivic origin. The two most important ingredients in the construction are the motivic effective slice filtration\cite{voevodsky2002openproblems} and Betti realization\cite{morel1999motivicbetti}.  

The motivic effective slice filtration gives a motivic effective slice tower and dually, a motivic effective slice cotower for each spectrum. These two towers are the analogue of the classical Whitehead tower and Postnikov tower. The associated spectral sequence of the towers, i.e. the motivic effective slice spectral sequence, turns out to be a powerful computational tool. For example, work by R\"{o}ndigs, Spitzweck and \O stv\ae r \cite{rondigs2019motivicsphere} gives the $E_1$-page of the motivic effective spectral sequence of the motivic sphere over any field and proves that it converges conditionally, in the sense of \cite{boardman1999convergence}, to the $\eta$-completed stable motivic homotopy groups where $\eta$ denotes the motivic Hopf map. 

On the other hand, the $C_2$-equivariant Betti realization $\ReB$ bridges the $\bbR$-motivic stable homotopy category and the $C_2$-equivariant stable homotopy category. This functor is studied in various sources, for example, in \cite{heller2016galoisequivariance} and \cite{bachmann2018betti}. Taking the Betti realization of a motivic tower, we obtain a $C_2$-equivariant tower which gives the $C_2$-effective spectral sequence.

 However, as Betti realization is a left Quillen adjoint functor (\cite{heller2016galoisequivariance}), it does not necessarily preserve homotopy limits. As a result, convergence of the $C_2$-effective spectral sequence for a general $C_2$-equivariant spectrum can be complicated.
We discuss the general convergence problem in Subsection \ref{subsec:esss} and show that, under certain assumptions, the convergence problem can be solved. These assumptions are satisfied in the case of $\kokq$, the main example of our paper.

The $C_2$-equivariant spectrum $\kokq$ is the $C_2$-equivariant Betti realization of the $\bbR$-motivic spectrum $\kq$, where $\kq$ denotes the very effective cover of the $(8,4)$-periodic Hermitian $K$-theory spectrum $\KQ.$ The spectrum $\kokq$ can be viewed as the $C_2$-equivariant analogue of the classical connective real $K$-theory spectrum $\mathbf{ko}$. Its mod $2$ cohomology is $A_{C_2}/\!\!/A_{C_2}(1)$, where $A_{C_2}$ and $A_{C_2}(1)$ denote the $C_2$-equivariant Steenrod algebra and its subalgebra generated by $\Sq^1, \Sq^2$. It also detects the first $C_2$-equivariant Hopf element which by abuse of notation is again denoted by $\eta$.

We prove the $C_2$-effective spectral sequence of $\kokq$ converges strongly in the sense of \cite{boardman1999convergence} to its $\eta$-completion. We describe the $E_1$-page of the $C_2$-effective spectral sequence for $\kokq$ in Theorem \ref{thm:koconverge}. The description is entirely explicit, but it requires some notation that will be introduced in Section \ref{esss} .

There is a $C_2$-equivariant map from $\kokq$ to $\ko$, the $C_2$-equivariant connective real $K$-theory spectrum. We prove that this map is an equivalence after $\eta$-completion (Proposition \ref{prop:redkqiskoceta}). 
As a corollary, if we further complete at $2$, the $(2,\eta)$-completion (equivalently, the $2$-completion. See Remark \ref{rmk:2etacpltfirst}) of $\kokq$ is equivalent to the $2$-completion of $\ko$ (Corollary \ref{prop:redkqiskoc2}). The homotopy ring of the latter is computed using the $C_2$-equivariant Adams spectral sequence in \cite{GHIR}. Our computation by the $C_2$-effective spectral sequence gives an independent method, and reproves the results.
It turns out that in the example of $\kokq$, the $C_2$-effective spectral sequence computation is much simpler than the $C_2$-equivariant Adams spectral sequence computation. First, the spectral sequence collapses after the $E_1$-page. Additionally, much of the homotopical structure is detected purely algebraically in the $E_1$-page: we don't need to use Toda brackets arguments to analyze the extension problems.

 Another advantage of the $C_2$-effective spectral sequence is its connection to other computational tools.  
 It is related closely to the motivic stable homotopy theory due to its motivic origin. At the same time, we can compare it to other $C_2$-equivariant computations and classical computations to produce new arguments or complementary results. For example, if we have assumed knowing the $C_2$-equivariant Adams computation in \cite{GHIR}, we can produce different proofs for the differentials and the extension problems in our computation. Another example is Lemma \ref{rholem}, where we compare $C_2$-equivariant and classical information to analyze hidden extensions.

There have been many computational tools established in the $C_2$-equivariant stable homotopy category which are formally similar, e.g. 
 the spectral sequences associated to the equivariant slice filtration (\cite{hill2016nonexistence},\cite{dugger1999postnikov}) and to the equivariant regular slice filtration (\cite{ullman2013slice}). We would like to point out that the $C_2$-effective spectral sequence is in general different from them. In fact, the $C_2$-equivariant slices for $\ko$ have not been completely identified, while its $C_2$-effective slices are known. The same is true for the $C_2$-equivariant sphere spectrum. A comparison of these spectral sequence can be found in \cite{heard2018equivariant}. 
% For example, for the $C_2$-equivariant sphere spectrum, its equivariant slices or regular slices are so far unknown. But its $C_2$-effective spectral sequence $E_1$-pages can be identified using motivic information.
 
  We have shown via the example of $\kokq$ that this spectral sequence allows us to give alternative computations that are independent and simpler, and we expect it will be useful for other $C_2$-equivariant computations as well. 

\subsection{Organization}
In Section \ref{esss}, we construct the $C_2$-effective spectral sequence, discuss the convergence problem, compute the $E_1$-page, and solve the convergence problem for $\kokq$. In Section \ref{EM}, we recall the RO($C_2$)-graded homotopy ring of $\HZt$ and use a $2$-Bockstein spectral sequence to compute the RO($C_2$)-graded homotopy ring of $\HZcplt.$ We carry out the computation of the $C_2$-effective spectral sequence for $\kokq$ in Section \ref{diff}, solve the extension problems in Section \ref{ext}, and attach the charts which depict the results in Section \ref{charts}.

\subsection{Notation and conventions}\hfill \\

%
%We implicitly take $2$-completion to make the results more comparable with \cite{GHIR} paper. The $2$-completion is not essential and the spectral sequence can be computed integrally. The integral computation is similar to the $2$-completed computation which we present in this paper. 

We employ the following grading conventions:
\begin{itemize}
  \item In the bidegree of the motivic sphere $S^{s,w}$, the degree $s$ is the topological degree and $w$ is the motivic weight. For example, the simplicial sphere is denoted by $S^{1,0}$ and the Tate sphere $\bbA^1-0$ by $S^{1,1}$.
  \item We bigrade the $C_2$-equivariant spheres according to the motivic convention. The trivial representation $1$ has bidegree $(1,0)$, while the sign representation $\sigma$ has bidegree $(1,1).$ We use the term coweight to refer to the number of copies of the trivial representation. Therefore, the equivariant degree $p+q\sigma$ is bigraded as $(p+q, q)$ with coweight $p.$
\end{itemize}

We employ notations as follows:
\begin{itemize}
	\item $\SH(k)$ is the motivic stable homotopy category over the base field $k$.
	\item $\SHeff(k)$ is the effective motivic stable homotopy category over the base field $k$.
    \item $\MZ$ (resp. $\MZt$) is the motivic Eilenberg-MacLane spectrum representing the motivic cohomology with $\bbZ$ (resp. $\bbF_2$) coefficients.
    \item $\kq$ is the very effective cover of the hermitian $K$-theory spectrum $\KQ$.
    \item $\kgl$ is the effective cover, or equivalently, the very effective cover of the algebraic $K$-theory spectrum $\KGL$.   
    \item $\HZ$ (resp. $\HZt$, $\HZcplt$) is the $C_2$-equivariant Eilenberg-MacLane spectrum representing $C_2$-equivariant cohomology with coefficients in the constant Mackey functor $\m{\bbZ}$ (resp. $\udl{\bbF_2}$, $\udl{\bbZ_2}$) Mackey functors.
    \item $\kR$ is the equivariant connective cover of the Atiyah's Real $K$-theory spectrum $\KR$.
    \item $\ko$ is the equivariant connective cover of the $C_2$-equivariant real $K$-theory spectrum $\KO$.
    \item $\kokq$ is the $C_2$-equivariant Betti realization of $\kq.$ It is equivalent to $\ko$ after $\eta$ completion (see Proposition \ref{prop:redkqiskoceta}).
    \item $\mathbf{ko}$ is the classical connective real $K$-theory spectrum.

\end{itemize}

\subsection{Acknowledgements}
First and foremost, the author would like to express her sincere gratitude to Dan Isaksen for suggesting this project, sharing his insights at various stages of the project, and checking all the proofs in this manuscript. The author would like to thank Peter May for his support, suggestions on organization, and for carefully reading several drafts of this paper. The author would also like to thank Zhouli Xu for his clear and detailed explanation about subtleties in extension problems. Thanks are also due to Mark Behrens, J.D. Quigley, Doug Ravenel, XiaoLin Danny Shi, Guozhen Wang, and Mingcong Zeng for many helpful conversations.

%Section 2

\section{The \texorpdfstring{$C_2$}{C2}-effective filtration}
\label{esss}

%%%%%%%%%%%%%%%%%%%%%%%%%%%%%%%%%%%%%%
\subsection{The motivic effective filtration}\hfill\\
\label{subsec:esss}

 In this subsection, we review motivic stable homotopy theory and the motivic effective slice filtration which was first proposed by Voevodsky in \cite{voevodsky2002openproblems}.
 
Let $\SH(k)$ denote the motivic stable homotopy category over the base field $k$.
Let $\SHeff(k)$ be the smallest full triangulated subcategory of $\Sh(k)$ that is closed under direct sums and
contains all suspension spectra of all smooth schemes. We define its $(p, q)$ suspension to be the subcategory: $$\Sigma^{p,q}\SHeff(k):=\{\Sigma^{p,q}X|X\in\SHeff \}.$$
Since $\SHeff(k)$ is closed under simplicial suspensions and simplicial desuspensions, we have an equivalence of subcategories $\Sigma^{p,q}\SHeff(k) \simeq \Sigma^{0,q}\SHeff(k)$. 
%Since $\SHeff(k)$ is closed under simplicial suspension and desuspension, the bigraded suspensions of $\SHeff(k)$ take the form $\Sigma^{0,q} \SHeff(k)$ for all $q \in \bbZ$.
These suspension subcategories form a motivic filtration, namely the motivic effective slice filtration:
\begin{equation*}
\cdots \subset \Sigma^{0,q+1}\SHeff(k)\subset \Sigma^{0,q}\SHeff(k)\subset \Sigma^{0,q-1}\SHeff(k) \subset \cdots .
\end{equation*} 

The inclusion 
$i_q: \Sigma^{0,q}\SHeff(k) \hookrightarrow \SH(k)$ 
preserves colimits. It has a right adjoint functor $r_q: \SH(k)\to \Sigma^{0,q}\SHeff(k)$. 

\begin{definition}
	Define the $q$-th motivic effective slice cover functor $f_q$ to be the composite 
	$$f_q:=i_q\circ r_q: \SH(k)\to \SH(k).$$
\end{definition}

The formal construction gives a sequence of natural transformations:
\begin{equation*}
\cdots \to f_{q+1} \to f_q  \to f_{q-1}  \cdots .
\end{equation*} 

\begin{definition}
	Let $E$ be any motivic spectrum. Define its motivic effective slice tower to be the following sequence of maps between motivic spectra:
$$\cdots \to f_{q+1}(E) \to f_q(E)  \to f_{q-1}(E)  \cdots .$$
\end{definition}

Dually, we have motivic effective slice cocovers and motivic effective slice cotowers.

\begin{definition}
Define the $q$-th motivic effective slice cocover to be the functor $f^q$ such that $f_{q+1}\to \id \to f^q$ forms a cofiber sequence.
\end{definition}
\begin{definition}

Let $E$ be any motivic spectrum. Define its motivic effective slice cotower to be the following sequence of maps between motivic spectra:
$$\cdots \to f^{q+1}(E) \to f^q(E)  \to f^{q-1}(E)\to  \cdots .$$
\end{definition}

\begin{definition}
	Define the $q$-th motivic effective slice functor $s_q$ such that $f_{q+1}\to f_q\to s_q$ forms a cofiber sequence.
\end{definition}

For any motivic spectrum $E$, its motivic effective slice tower has associated graded slices $s_*(E).$ By passing to motivic homotopy groups, we obtain a spectral sequence whose $E_1$-page is the homotopy groups of these slices:
$$E_1^{s,q,w}=\pi_{s,w}s_q(E)\Rightarrow \pi_{s,w}(E) .$$

%%%%%%%%%%%%%%%%%%%%%%%%%%%%%%%%%%%%%%

\subsection{The \texorpdfstring{$C_2$}{C2}-effective spectral sequence}\hfill\\

The motivic effective slice filtration is defined for general base fields. 
When the base field $k$ admits an embedding to $\bbC$, we have a Betti realization functor $\ReG:\SH(k)\to \SH$ by taking the $\bbC$-points of smooth schemes(\cite{morel1999motivicbetti}). When $k$ further admits an embedding to $\bbR$, the conjugation on complex numbers induces a $C_2$-action on the target of the Betti realization functor. In such cases, we get a $C_2$-equivariant Betti realization functor $$\ReB:\SH(k)\to \SH_{C_2}.$$ We restrict our attention to the base fields $\bbR$ and $\bbC$. In fact, we have the following commutative diagram of categories and functors:
\begin{equation}
\label{magicsquare}
\begin{tikzcd}
\SH(\bbR)\ar[r,"\ReB"]\ar[d,"i^*"] & \SH_{C_2}\ar[d, "U"] \\	
\SH(\bbC)\ar[r,"\ReG"] & \SH .
\end{tikzcd}
\end{equation}
The functor $i^*$ is induced by the embedding $i:\bbR\to\bbC$, and $U$ is the forgetful functor.

The $C_2$-equivariant Betti realization functor is well studied in previous work. We record some of the relevant properties as shown in \cite[Prop 4.8, 4.17]{heller2016galoisequivariance}. 

\begin{prop}(Heller-Ormsby)
\label{prop:betti}
The $C_2$-Betti realization functor is strong symmetric monoidal, left Quillen, and it satisfies that 
\begin{itemize}
	\item $\ReB(\MZ)\simeq \HZ$.
	\item $\ReB(\MZt)\simeq \HZt$.
	\item $\ReB(S^{a,b})\simeq S^{a,b}$.
\end{itemize}	
\end{prop}

Let $X$ be a $C_2$-equivariant spectrum and suppose that there exists a motivic spectrum $\overline{X}$ in $\SH(\bbR)$ such that $X\simeq \ReB(\overline{X})$. We apply $\ReB$ to the motivic effective slice cotower (resp., tower) of $\overline{X}$ to obtain a cotower (resp., tower) of equivariant spectra for $X$.

\begin{defn}
	Define the $C_2$-effective cotower (resp., tower) of $X$ associated to $\overline{X}$ to be the $C_2$-equivariant Betti realization of the motivic effective slice cotower (resp., tower) of $\overline{X}.$ 
	$$f^*_{\overline{X}}X:=\ReB(f^*\overline{X}),$$ 
	$$f_*^{\overline{X}}X:=\ReB(f_*\overline{X}),$$ 
	Define the $C_2$-effective slices of $X$ associated to $\overline{X}$ to be the fiber in the following cofiber sequence:
	$$s^{\overline{X}}_q(X)\to f_{\overline{X}}^q(X)\to f_{\overline{X}}^{q-1}(X),$$
	or equivalently, the cofiber in the following cofiber sequence:
	$$f^{\overline{X}}_{q+1}(X)\to f^{\overline{X}}_{q}(X)\to s^{\overline{X}}_q(X).$$
\end{defn}

\begin{rmk}
	In the above definition, the construction of the cotower and tower depends on the choice of $\overline{X}$. When the choice of $\overline{X}$ is clear in the context, we omit $\overline{X}$ from the notation, writing $f^*(X)$ instead of $f^*_{\overline{X}}(X)$ for example.
\end{rmk}

\begin{rmk}
As mentioned in the introduction, for general $C_2$-equivariant spectra, the $C_2$-equivariant effective slices are different from the slices under the filtrations defined in \cite{dugger1999postnikov}, \cite{hill2016nonexistence} or \cite{ullman2013slice}. Heard proved that when certain assumptions are satisfied, the motivic effective slices realize to the $C_2$-equivariant slices\cite[Theorem 5.15]{heard2018equivariant}. We consistently use the adjective 'effective' to distinguish our spectral sequence from these other spectral sequences that are formally similar.  
\end{rmk}
%TODO sub/sup

Since $\ReB$ preserves cofiber sequences, the $C_2$-effective slices $s_q(X)$ are equivalent to $\ReB(s_q(\overline{X})).$ 

\begin{defn}
	Define the effective completion of $X$ associated to $\overline{X}$ to be the homotopy limit of the $C_2$-effective cotower associated to $\overline{X}$, i.e.
	$$ec_{\overline{X}}(X):=\holim_q f_{\overline{X}}^{q}(X).$$
\end{defn}
%TODO [changed]difference q-1 and q?
Passing to homotopy groups, we have a spectral sequence of the form:
$$E^1_{s,q,w}=\pi^{C_2}_{s,w}(s_q^{\overline{X}}(X))\implies \pi^{C_2}_{s,w}(ec_{\overline{X}}(X)),$$
for a $C_2$-equivariant spectrum $X$ such that $X\simeq \ReB(\overline{X}).$

\begin{rmk}
	A spectral sequence associated to a Postnikov-like tower (which in our case is referred to as cotower) should compute the homotopy limit of the tower. 	In the motivic effective slice context, we have the notion of the 'slice completion' of a motivic spectrum $\overline{X}$, denoted by $sc(\overline{X})$. It is defined as the homotopy limit of the motivic effective slice cotower: $$sc(\overline{X}):=\holim_q f^{q}(\overline{X}).$$ 
	%TODO [changed]difference q-1 and q?
	In general, a motivic spectrum is not necessarily equivalent to its slice completion. When it is, we say that the motivic spectrum is slice complete. 
	Similarly, a $C_2$-equivariant spectrum is not necessarily equivalent to its effective completion.
	Since the Betti realization functor is a left Quillen adjoint (Proposition \ref{prop:betti}), it does not necessarily preserve homotopy limits. As a result, the effective completion $ec_{\overline{X}}(X)$ is not necessarily equivalent to $\ReB(sc(\overline{X}))$.
\end{rmk}

\subsection{\texorpdfstring{$C_2$}{C2}-effective completions}

We now prove several propositions about the effective completion which will be used later in the paper regarding the $C_2$-effective spectral sequence of $\kokq.$ 

\begin{definition}
	Let $X$ be a $C_2$-equivariant spectrum. We say $X$ is effective complete, if there exists a $\bbR$-motivic spectrum $\overline{X}$ such that the following are satisfied:
	\begin{enumerate}
  \item $\ReB(\overline{X})\simeq X.$
  \item The canonical map $\overline{X}\to \holim_q f^{q}(\overline{X})$ is an equivalence, i.e. the motivic spectrum $\overline{X}$ is slice complete.
  %TODO [changed]difference q-1 and q?
  \item The canonical map $X\to ec_{\overline{X}}(X)$ is an equivalence.
\end{enumerate}

\end{definition}
We show that effective completeness has the 2 out of 3 property. 
\begin{prop}
	Let $A\to B\to C$ be a sequence of maps of $C_2$-equivariant spectra such that it is the image of the Betti realization of a $\bbR$-motivic cofiber sequence
	$\overline{A}\to\overline{B}\to\overline{C}.$
	Among the following three conditions, any two of them implies the third.
	\begin{itemize}
		\item $A$ is effective complete with respect to $\overline{A}$.
		\item $B$ is effective complete with respect to $\overline{B}$.
		\item $C$ is effective complete with respect to $\overline{C}$.
	\end{itemize}
	\label{closelemma}
\end{prop}

\begin{proof}

We consider the following commutative diagram:
%$$sc_{\overline{A}}(A)\to sc_{\overline{B}}(B)\to sc_{\overline{C}}(C). $$ 
$$
\xymatrix{
	A \ar[r]\ar[d]& B\ar[d]\ar[r] & C\ar[d]\\
	ec_{\overline{A}}(A)\ar[r] & ec_{\overline{B}}(B)\ar[r] & ec_{\overline{C}}(C)
}.$$

The bottom row is the homotopy limit of the sequence of the $C_2$-effective cotowers 
$$ f_{\overline{A}}^*(A) \to f_{\overline{B}}^*(B) \to f_{\overline{C}}^*(C).$$
Since Betti realization and all $f^q$ preserve cofiber sequences, this sequence of cotowers has each level as a cofiber sequence. 
Passing to homotopy limit, we get that the bottom row forms a cofiber sequence. Since the top row is also a cofiber sequence, we conclude the 2 out of 3 property.
\end{proof}

By the above proposition, we have that effective completeness is closed under taking fibers, cofibers, extensions and suspensions.

The following result is straightforward by definition.
\begin{prop}
\label{slicearecplt}
	Let $X$ be a $\bbR$-motivic spectrum. For any $q\in \bbZ$, $\ReB(s_q(X))$ (resp., $\ReB(f^q(X))$) is effective complete with respect to $s_q(X)$ (resp., $f^q(X)$).
\end{prop}

\begin{definition}
	Let $\eta\in\pi_{1,1}(S^{0,0})$ be the first $\bbR$-motivic (and $C_2$-equivariant) Hopf map.
\end{definition}

\begin{rmk}
	Let $\bbA^1$ be the affine line and $\bbP^1$ be the projective line. The first motivic Hopf map $\eta$ is the projection $\bbA^2-0\to \bbP^1,$ i.e., a map $S^{3,2}\to S^{2,1}.$ 
	
	The $ C_2$-equivariant Hopf map is the projection map:
	$$\bbC^2-0\to \bbC\bbP^1: (x,y)\mapsto [x,y],$$
	where the $C_2$-action is given by the complex conjugation.
	Since $\bbC^2-0\cong S^{1+2\sigma}$ and $S^{1+\sigma}\cong \bbC\bbP^1$, this is also a map $S^{3,2}\to S^{2,1}$ under our bigrading convention.
	
	In fact, we can define the first Hopf element in $\SH(\bbR)$, $\SH(\bbC)$, $\SH_{C_2}$ and $\SH$. These definitions are compatible with the functors in Diagram \ref{magicsquare}. We abuse notation and use the same name $\eta$ for all these maps. 
	
\end{rmk}

\begin{definition}
	Let $X$ be an $\bbR$-motivic ($C_2$-equivariant) spectrum. Define the $\eta$-completion of $X$ to be 
	$$X^{\wedge}_\eta:=\holim_q X/\eta^q.$$ 
	We say $X$ is $\eta$-complete if the map 
	$X\to X^\wedge_\eta $ is a weak equivalence.
\end{definition}

\begin{rmk}
	Suppose we have an $\bbR$-motivic or $C_2$-equivariant cofiber sequence $A\to B\to C$. We consider the diagram 
	$$
	\begin{tikzcd}
	A\ar[r]\ar[d] & B\ar[r]\ar[d] & C\ar[d]\\
	A^\wedge_\eta\ar[r] & B^\wedge_\eta\ar[r] & C^\wedge_\eta 
	\end{tikzcd}
	$$
	We can prove similarly as Proposition \ref{closelemma} that $\eta$-completeness has the 2 out of 3 property. It is closed under taking fibers, cofibers, extensions and suspensions.
	\label{closelemmaeta}
\end{rmk}

\begin{prop}
	Let $X$ be a $C_2$-equivariant spectrum and let $\overline{X}$ be an $\bbR$-motivic spectrum such that $X\simeq \ReB(\overline{X})$. If $X\simeq f_n^{\overline{X}}(X)$ for some $n$, then $ec_{\overline{X}}(X)$ is $\eta$-complete. 	
\label{etacplt}
\end{prop}

\begin{proof}
%The Hopkins-Morel-Hoyois Theorem\cite{hoyois2015algebraic} identifies the motivic zeroth slice $s_0(S^{0,0})$ of the motivic sphere spectrum with $\MZ$. 
By \cite{rondigs2019motivicsphere}, the motivic zeroth slice $s_0(S^{0,0})$ of the motivic sphere spectrum is equivalent to $\MZ$. Therefore each motivic effective slice $s_q(\overline{X})$ is an $\MZ$-module. Applying $C_2$-equivariant Betti realization, we have that each $C_2$-effective slice $s_q(X)$ is an $\HZ$-module.
We see from the homotopy ring of $\HZ$ as computed in \cite{hill2016nonexistence} that the element $\eta$ acts trivially on $\HZ$. Therefore it also acts trivially on all $\HZ$-modules, in particular, all $C_2$-effective slices $s_q(X)$. As a result, all $C_2$-effective slices are $\eta$-complete.

By the assumption, the $C_2$-effective cotower of $X$ has the form:
$$\cdots\to f^{n+2}_{\overline{X}}(X)\to f^{n+1}_{\overline{X}}(X)\to 0\to 0 \cdots.$$
The $(n+1)$-st effective cocover $f^{n+1}_{\overline{X}}(X)$ is equivalent to the $(n+1)$-st effective slice $ s^{\overline{X}}_{n+1}(X)$ which is $\eta$-complete by the discussion above. Since $\eta$-completeness is closed under extensions by Remark \ref{closelemmaeta}, every effective cocover $f^{q}_{\overline{X}}(X)$ is $\eta$-complete. As a result, we have 
 \begin{align*}
 	ec_{\overline{X}}(X)&=\holim_q f^{q}_{\overline{X}}(X) \simeq \holim_q \holim_k f^{q}_{\overline{X}}(X)/\eta^k \\
 	&\simeq \holim_k \holim_q f^{q}_{\overline{X}}(X)/\eta^k\simeq \holim_k ec_{\overline{X}}(X)/\eta^k\simeq ec_{\overline{X}}(X)^{\wedge}_\eta.
 \end{align*}
\end{proof}
%TODO [changed]difference q-1 and q?

\begin{prop} 
\label{isoeta}
Let $X$ be a $C_2$-equivariant spectrum and let $\overline{X}$ be an $\bbR$-motivic spectrum such that $X\simeq \ReB(\overline{X})$. For any positive integer $k$, there is a canonical equivalence:
$$ec_{\overline{X}/\eta^k}(X/\eta^k)\simeq ec_{\overline{X}}(X)/\eta^k.$$

\end{prop}

\begin{proof}
	We compare $f^{q}(\overline{X}/\eta^k)$ and $f^{q}(\overline{X})/\eta^k$ by 
	considering the following commutative diagram:
		\begin{equation*}
\xymatrix{
\Sigma^{k,k} \overline{Y_k^q} \ar[d]\ar[r] &  \Sigma^{k,k} f^{q}(\overline{X}) \ar[d]^{\eta^k} \ar[r] & {\Sigma^{k,k}} f^{q-k}(\overline{X})\simeq f^{q}(\Sigma^{k,k}\overline{X})\ar[d]^(.5){f^q(\eta^k)}\\
{*}\ar[d]\ar[r] & f^{q}(\overline{X})\ar[r]^{=}\ar[d]&  f^{q}(\overline{X})\ar[d]\\
\Sigma^{k+1,k} \overline{Y_k^q} \ar[r] &  f^{q}(\overline{X})/\eta^k\ar[r] & f^{q}(\overline{X}/\eta^k)
}.
	\end{equation*}
	The middle and the right column are cofiber sequences that define $f^{q}(\overline{X})/\eta^k$ and $f^{q}(\overline{X}/\eta^k)$. The left column is obtained by taking the fiber in each row. The spectrum $\overline{Y_k^q}$ denotes the fiber of the map $f^{q}(\overline{X})\to f^{q-k}(\overline{X}),$ which can be identified with the limit of the truncated cotower of $\overline{X}$.
	
%	Therefore $f^q(\overline{X})/\eta^k$ is an extension of $\Sigma^{2,1}s_q(\overline{X})$ and $f^{q}(\overline{X}/\eta)$.  
	For all integers $q$, we apply Betti realization to the bottom row. Since Betti realization preserves cofiber sequences, 
	we get a levelwise cofiber sequence of towers:
			$$\xymatrix{
			\ar[d] &\ar[d] & \ar[d]\\
			\Sigma^{k+1,k}Y_k^{q+1}\ar[d]\ar[r]&f^{q+1}(X)/\eta^k\ar[r]\ar[d]& f^{q+1}(X/\eta^k)\ar[d]\\
\Sigma^{k+1,k}Y_k^q\ar[r]\ar[d]&f^q(X)/\eta^k\ar[r]\ar[d]& f^{q}(X/\eta^k)\ar[d]\\
& & \\
	 }
	$$
	where $Y_k^*$ denotes the image of $\overline{Y_k^*}$ under realization.
	By diagram chasing, the composite of any $k$ adjacent maps in the first column is trivial, i.e. the composite $Y^{q+k}\to Y^{q+k-2}\to \dots \to Y^{q}$ is trivial for all $q.$
On the other hand, the maps in the second and the third columns are from the $C_2$-effective cotowers. 
	Taking homotopy limits, we get the desired isomorphism. 
	
%TODO [changed]comma
\end{proof}

\begin{prop}
	If $X/\eta$ is effective complete with respect to $\overline{X}/\eta$,  and $ec_{\overline{X}}(X)$ is $\eta$-complete, then $ec_{\overline{X}} (X)\simeq  X^\wedge_\eta.$
	\label{etaconv}
\end{prop}

\begin{pf}
	By the assumption that $ec_{\overline{X}}(X)$ is $\eta$-complete, it is sufficient to show that $X^\wedge_\eta$ is equivalent to $ ec(X)^\wedge_\eta$. It is further sufficient to show that $X/{\eta^k}\to ec(X)/\eta^k$  is an equivalence for all $k\geq 1$.

	Note that $X/\eta^k$ is effective complete for any $k\geq 1$. In fact, when $k=1$, it is the assumption. For $k\geq 2$, it follows by induction on $k$ from Lemma \ref{closelemma} and the cofiber sequence 
	$$X/\eta\to X/\eta^k\to X/\eta^{k-1}.$$
	
	By Lemma \ref{isoeta}, we have that $ec(X)/\eta^k\cong ec(X/\eta^k).$ Therefore, for any $k\geq 1$, we have
	$$X/{\eta^k}\simeq ec(X/\eta^k)\cong ec(X)/\eta^k.$$
\end{pf}
	
%%%%%%%%%%%%%%%%%%%%%%%%%%%%%%%%%%%%%%
\subsection{\texorpdfstring{$\ko$}{koc2} and the realization of \texorpdfstring{$\kq$}{kq}}\hfill\\

Let $\kq$ denote the very effective cover of the Hermitian $K$-theory spectrum $\KQ$ (\cite{horn2005representability}).
The main example of this paper is the realization of $\kq$. 

\begin{defn}
	Let $\kokq$ denote the $C_2$-equivariant Betti realization of $\kq$.
\end{defn}

The spectrum $\kokq$ is closely related to $\ko,$ the connective cover of the $C_2$-equivariant $K$-theory spectrum $\KO$ (\cite{GHIR}). In this section, we discuss their connection.

In the classical stable homotopy category, we have the Wood cofiber sequence:
$$\Sigma^1\mathbf{ko}\xrightarrow{\eta} \mathbf{ko} \to \mathbf{ku}.$$

Its $C_2$-equivariant analogue is the cofiber sequence (\cite{GHIR}):
\begin{equation}
\label{equa:eqwood}
	\Sigma^{1,1}\ko\xrightarrow{\eta} \ko \to \kR,
\end{equation}
where $\kR$ is the connective cover of the Atiyah's Real $K$-theory spectrum $\KR$.

Similarly, in the motivic stable homotopy category, we also have a Wood cofiber sequence (see \cite{ananyevskiy2019hermitiank} ):
\begin{equation}
\label{equa:motwood}
	\Sigma^{1,1}\kq\xrightarrow{\eta} \kq \to \kgl,
\end{equation}

where $\kgl$ is the very effective cover of the motivic algebraic $K$-theory spectrum $\KGL$.

The question is whether the $C_2$-equivariant realization takes the motivic Wood cofiber sequence to the $C_2$-equivariant one. We show this is true for the periodic case.
\begin{prop}
The $C_2$-equivariant Betti realization of the periodic $\bbR$-motivic Wood cofiber sequence $\Sigma^{1,1}\KQ\xrightarrow{\eta} \KQ \to \KGL$ is the periodic $C_2$-equivariant Wood cofiber sequence 
$\Sigma^{1,1}\KO\xrightarrow{\eta} \KO \to \KR$.
\label{prop:woodcofsequences}
\end{prop}
\begin{pf}
	Betti realization preserves cofiber sequences and takes the $\bbR$-motivic element $\eta$ to the $C_2$-equivariant element $\eta$. Therefore it is sufficient to show that $\ReB(\KQ)\simeq \KO.$
	
	By \cite[Theorem 2.2]{karoubi2016realwitt}, there is a natural isomorphism $\ReB(\KQ)_0(-)$ and $(\KO)_0(-)$. This extends to an isomorphism between the periodic cohomology theories which gives an equivalence between the representing spectra.
\end{pf}

\begin{rmk}
	It follows by the proof of Proposition \ref{prop:woodcofsequences} that $\ReB(\KGL)\simeq \KR$. Another proof can be found in \cite[Example 5.17]{heard2018equivariant}.
\end{rmk}

In general, the very effective cover of a motivic spectrum $X$ does not necessarily realize to the $C_2$-equivariant connective cover of its $C_2$-equivariant Betti realization $\ReB{X}.$ Therefore, we cannot deduce directly from the periodic case that the motivic cofiber sequence (\ref{equa:motwood}) realizes to the equivariant cofiber sequence (\ref{equa:eqwood}). In fact, though we know that the realization $\ReB(\kgl)$ is equivalent to $\kR$ by \cite[Theorem 5.15, Example 5.17]{heard2018equivariant}, it is not known yet whether the realization $\ReB(\kq)$ is equivalent to $\ko$. However, we have a map
$\phi:\kokq\to \ko$ by applying the connective cover functor to the equivalence $\ReB{\KQ}\xrightarrow{\simeq} \KO$ and using the universal property. We show that this map $\phi$ is an equivalence after $\eta$-completion.

%It is not known whether this map is an equivalence. However, we know that the map $\phi$ induces isomorphism on $\HZt_{*,*}$ homology by \cite[Theorem 10.15]{GHIR} and \cite{ananyevskiy2019hermitiank}. Therefore, $\kokq$ and $ \ko$ have isomorphic Adams spectral sequences.
%By \cite[Corollary 6.47]{HK} the $C_2$-equivariant Adams spectral sequence converges to $2$-completions. Thus we have:
\begin{prop}
\label{prop:redkqiskoceta}
The map $\phi$ induces an equivalence 
$$(\kokq)^\wedge_\eta \simeq (\ko)^\wedge_\eta $$	
of $C_2$-equivariant spectra after $\eta$-completion.
\end{prop}

\begin{pf}

In general, if we have a cofiber sequence of $C_2$-equivariant spectra $\Sigma^{1,1}A\xrightarrow{\eta} A\to B$,
the exact couple
$$
\begin{tikzcd}
\pi_{*-1,*-1}A \ar[rr,"\eta"]& & \pi_{*,*}A\ar[ld]\\	
& \pi_{*,*} B\ar[lu] & 
\end{tikzcd}
$$
gives a convergent $\eta$-Bockstein spectral sequence of the form:
$$\pi_{*,*}B [\eta] \implies \pi_{*,*}A^\wedge_\eta.$$

Consider the commutative diagram:
$$
\begin{tikzcd}
\Sigma^{1,1}\kokq\ar[d,"\Sigma^{1,1}\phi"]\ar[r,"\eta"] & \kokq\ar[d,"\phi"]\ar[r] & \kR\ar[d,"\simeq"] \\	
\Sigma^{1,1}\ko\ar[r,"\eta"] & \ko\ar[r] & \kR
\end{tikzcd},
$$
where the top level is the realization of the motivic Wood cofiber sequence and the bottom level is the $C_2$-equivariant Wood cofiber sequence. 
The top and the bottom levels give two isomorphic $\eta$-Bockstein spectral sequences. Therefore, the map $\phi$ induces a $\pi_{*,*}$ equivalence. The result follows.
%TODO [changed, added before]reference
\end{pf}

	\begin{rmk}
	\label{rmk:2etacpltfirst}
	In general, for a $C_2$-equivariant spectrum $X$, its $2$-completion $X^\wedge_2$ is equivalent to its $\HZt$-nilpotent completion (\cite[Corollary 6.47]{HK}). By Proposition \ref{slicearecplt}, $\HZt$ is $\eta$-complete. Therefore, the $2$-completed spectrum $X^\wedge_2$ is already $\eta$-complete.
	\end{rmk}

\begin{cor}
\label{prop:redkqiskoc2}
The map $\phi$ induces an equivalence 
$$(\kokq)^\wedge_2 \simeq (\ko)^\wedge_2 $$	
of $C_2$-equivariant spectra after $2$-completion.
\end{cor}

\begin{proof}
	Since $\phi^\wedge_\eta$ is an equivalence,  the $2$-completion $(\phi^\wedge_\eta)^\wedge_2\simeq \phi^\wedge_2$ is also an equivalence.
\end{proof}

%The $C_2$-equivariant spectrum $\ko$ is the connective cover of the $C_2$-equivariant $K$-theory spectrum $\KO$. The motivic counterpart of $\KO$ is the (8,4)-periodic Hermitian $K$-theory motivic spectrum $\KQ$ which detects Karoubi’s Hermitian $K$-theory\cite{horn2005representability}. In fact $\ReB(\KQ)\simeq\KO.$\mysidetodo{ref.Asked Drew Heard} 
%The connective version of $\KQ$, which is denoted by $\kq$, is its very\mysidetodo{introduce very effective?} effective cover(see \cite{ananyevskiy2019hermitiank}). 
%Work by Heard(\cite{heard2018equivariant}) proved that the very effective cover realizes to the $C_2$-connective cover. 
%The motivic spectrum $\kq$ is the counterpart of $\ko$ in the sense that $\ReB(\kq)\simeq\ko$.	

%%%%%%%%%%%%%%%%%%%%%%%%%%%%%%%%%%%%%%
\subsection{Motivic effective slices of \texorpdfstring{$\kq$}{kq} and the \texorpdfstring{$C_2$}{C2}-effective slices of \texorpdfstring{$\kokq$}{kokq}}\hfill\\

In this section, we apply the general theory of $C_2$-effective filtrations to the $C_2$-equivariant spectrum $\kokq.$ The construction of the $C_2$-effective spectral sequence requires some understanding of the motivic effective slice tower of $\kq.$ Work by R{\"o}ndigs and {\O}stv{\ae}r (\cite{rondigs2016slices}, \cite{rondigs2016multiplicative}) gives the motivic effective slices of $\KQ$, the multiplicative relation and the differentials. The results for $\KQ$ translate to $\kq$. Later work by R{\"o}ndigs, Spitzweck and {\O}stv{\ae}r (\cite{rondigs2018hopfmap}) showed the effective spectral sequence converges strongly:
$$\pi_{*,*}s_*(\kq)\Rightarrow \pi_{*,*}({\kq}^{\wedge}_\eta).$$ 
%The computation of this spectral sequence is carried out in \cite{ananyevskiy2019hermitiank} for a general base field in certain bidegrees. 

We summarize the known results for $\kq.$
%TODO despcription

 The following formula quoted from \cite{ananyevskiy2019hermitiank} describes the motivic effective slices of $\kq$.
\begin{equation*}
    s_*(\kq)\simeq  \MZ[\eta,\sqrt\alpha]/(2\eta, \eta^2 \xrightarrow{\delta}\sqrt\alpha).
\end{equation*}

We explain this expression.

Additively, the slices of $\kq$ are wedge sums of suspensions of Eilenberg-Maclane spectra. All degrees can be described at once as indexed over the polynomial algebra $\bbZ[\eta,\sqrt\alpha]/2\eta$.
	
The generator $\eta$ represents a summand $\Sigma^{1,1}\MZt$ in $s_1(\kq)$ and $\sqrt\alpha$ represents a summand $\Sigma^{4,2}\MZ$ in $s_2(\kq)$. In general, $\eta^p(\sqrt\alpha)^q$ contributes a summand $\Sigma^{p+4q,p+2q}\MZt$ to $s_{p+2q}(\kq)$ when $p>0,$ and  $(\sqrt\alpha)^q$ contributes a summand $\Sigma^{4q,2q}\MZ$ to $s_{2q}(\kq).$ 
\begin{eg}
\label{eg:slices}
 We give three examples of motivic effective slices of $\kq$.
 \begin{enumerate}
 	\item The zeroth slice $s_0(\kq)$ is $\MZ$ and is indexed by $1$.
 	\item The first slice $s_1(\kq)$ is $\Sigma^{1,1}\MZt$ and is indexed by $\eta$.
 	\item The second slice $s_2(\kq)$ is $\Sigma^{2,2}\MZt\vee \Sigma^{4,2}\MZ$. The first wedge summand is indexed by $\eta^2$ and the second by $\sqrt\alpha.$
 \end{enumerate}	
\end{eg}

However, the polynomial algebra $\bbZ[\eta,\sqrt\alpha]/2\eta$ does not accurately describe the multiplicative structure. There is an extra multiplicative relation which is represented by the map $\eta^2 \xrightarrow{\delta}\sqrt\alpha$ in the notation of Ananyevskiy, R\"ondigs and \O stv\ae r. Conceptually, this map describes a relation between the wedge summands indexed by $\eta^2$ and by $\sqrt\alpha$ in $s_2(\kq)$. We explain what this map means. 

%The source $\eta\cdot \eta$ stands for the smash product of the first slice with itself $s_1(\kq)\wedge s_1(\kq)$. The target $\alpha$ stands for the wedge summand $\Sigma^{4,2} \MZ$ in the second slice $s_2(\kq)$ that is indexed by $\alpha.$

Consider the multiplication map $s_1(\kq)\wedge s_1(\kq)\to s_2(\kq).$ Using the descriptions of the first and the second slices in Example \ref{eg:slices}, this is a map 
\begin{equation}
\label{multi}
\mu:\Sigma^{1,1} \MZt \wedge \Sigma^{1,1} \MZt \simeq s_1 \kq \wedge s_1 \kq \to s_2 \kq \simeq \Sigma^{2,2} \MZt \vee \Sigma^{4,2} \MZ.
\end{equation}

%we consider the following composition:
%\begin{equation}
%\label{multi}
%\Sigma^{1,1} \MZt \wedge \Sigma^{1,1} \MZt \simeq s_1 \kq \wedge s_1 \kq \to s_2 \kq \to \Sigma^{4,2} \MZ
%\end{equation}
%where the first map is the multiplication map, and the second map is projection to the summand indexed by $\alpha.$
By \cite{hoyois2017steenrod}, as an $\MZt$-module, the source of the map (\ref{multi}) decomposes into a wedge sum of suspensions of $\MZt$. This wedge sum is indexed over the set of degrees of additive generators in the motivic Steenrod algebra. We denote this indexing set by $I$. 

Explicitly, we have a decomposition of the source:
\begin{equation}
\Sigma^{1,1}\MZt \wedge \Sigma^{1,1}\MZt\simeq \Sigma^{2,2}\MZt \vee \Sigma^{3,2}\MZt \vee_{(i,j)\in I\symbol{92}\{(0,0),(1,0)\}} \Sigma^{i+2,j+2}\MZt.
\label{decomposition}
\end{equation}

In the above expression, we pick out the wedge summands indexed by $(0,0)$ and $(1,0)$. This is because that under the decomposition (\ref{decomposition}), the map (\ref{multi}) is trivial for degree reasons on all other wedge summands except these two. When restricted to the summand indexed by $(0,0)$, the map $\mu|_{(0,0)}: \Sigma^{2,2}\MZt\to \Sigma^{2,2} \MZt \vee \Sigma^{4,2} \MZ$ is $(\id,0).$
%identity on the first component and trivial on the second component. 
When restricted to the summand indexed by $(1,0)$, the map $\mu|_{(1,0)}: \Sigma^{3,2}\MZt\to \Sigma^{2,2} \MZt \vee \Sigma^{4,2} \MZ$ is $(0, \Sigma^{3,2}\delta)$. Here $\delta$ denotes the natural Bockstein connecting map $\MZt\to \Sigma^{1,0}\MZ$.\\

\begin{comment}

In the effective spectral sequence of $\kq$, this multiplicative relation gives a relation: 
$$ah_1\cdot bh_1=ab h_1^2 +\Sq^1(a)\Sq^1(b) v,$$
where $a,b$ are classes in $\MZt_{*,*}$ and $ah_1, bh_1$ are classes in $ \MZt_{*,*}[\eta]$. Therefore, the $E_1$-page with the multiplicative relation can be described as following:
\begin{equation*}
    E^1_{*,*,*}\cong \MZ_{*.*}[h_1,v_1^2]/(2h_1, ah_1\cdot bh_1=ab h_1^2 +\Sq^1(a)\Sq^1(b) v_1^2),
\end{equation*}

where $h_1$ represents the unit element in $\MZt_{*,*}[\eta]$ and is in tri-degree $(1,1,1)$, and $v_1^2$ represents $\MZ_{*,*}[\alpha]$ and is in tri-degree $(4,2,2).$
\end{comment}

%\subsection{$C_2$-equivariant effective spectral sequence of $\ko$}

%The realization functor brings $\kq$ to $\ko$. It further preserves suspensions, and maps motivic Eilenberg-Maclane spectra to equivariant Eilenberg-Maclane spectra associated to constant Mackey functor. 

Applying the $C_2$-equivariant Betti realization functor to the motivic effective slices of $\kq$, we have the following description of the $C_2$-effective slices of $\kokq$:
%The $C_2$-effective slices of $\ko$ associated to $\kq$ have the following form:
\begin{equation}
 \label{koc2e1spec}
    s_*(\kokq) \simeq \HZ[\eta,\sqrt\alpha]/(2\eta, \eta^2\xrightarrow{\delta}\sqrt\alpha).
\end{equation}

The expression (\ref{koc2e1spec}) has the same meaning as in the motivic case. The element $\eta$ indexes a wedge summand of suspension of $\HZt$ and the element $\sqrt\alpha$ indexes a wedge summand of suspension of $\HZ.$ The map $\delta$ describes the multiplicative structure in $s_1(\kokq)\wedge s_1(\kokq)\to s_2(\kokq).$

%%%%%%%%%%%%%%%%%%%%%%%%%%%%%%%%%%%%%%

%%%%%%%%%%%%%%%%%%%%%%%%%%%%%%%%%%%%%%
\subsection{The effective completion of \texorpdfstring{$\kokq$}{koc2}}\hfill\\

In this section, we show that the $C_2$-effective spectral sequence for $\kokq$ actually converges to the homotopy of the $\eta$-completion of $\kokq$. We fix the preferred choice of the Betti realization preimage to be $\kq.$

%TODO [changed]added preferred choice
%The proof involves the Atiyah Real $K$-theory spectra $\kR$ and its motivic counterpart $\kgl$ which is defined as the effective cover of the $(2,1)$-periodic algebraic $K$-theory motivic spectrum $\KGL$. By \cite[Example 5.17]{heard2018equivariant}, the Betti realization of $\kgl$ is $\kR$.

\begin{thm}
	The effective completion of $\kokq$ is equivalent to the $\eta$-completion of $\kokq$, i.e.
	$$(\kokq)_{\eta}^{\wedge}\simeq ec_{\kq}(\kokq).$$
\end{thm}
\begin{pf}
	By Proposition \ref{etacplt}, we have that $ec(\kokq)$ is $\eta$-complete. Therefore by Proposition \ref{etaconv}, it is sufficient to show that $\kokq/\eta$ is effective complete. By \cite{ananyevskiy2019hermitiank}, we have a cofiber sequence 
	$$\Sigma^{1,1}\kq\xrightarrow{\eta} \kq \to \kgl.$$ Taking Betti realization, we identify the cofiber $\kokq/\eta$ with $\ReB(\kgl)\simeq\kR.$ By \cite[Theorem 5.15, Example 5.17]{heard2018equivariant}, the effective slice tower of $\kgl$ realizes to the $C_2$-equivariant slice tower of $\kR$ whose homotopy limit is contractible(\cite{dugger1999postnikov}). Dually, the homotopy limit of the $C_2$-effective cotower is equivalent to $\kR$, i.e. the spectrum $\kR\simeq \kokq/\eta$ is effective complete. 	
%	The result follows from Proposition \ref{etaconv} and Proposition \ref{koetaec}.	
\end{pf}

As a result, we have the following:

%\begin{thm}
%\label{thm:koconverge}
%	There is a strongly convergent $C_2$-effective spectral sequence with the following form:
%	$$\pi_{s,w}(\HZ[\eta,\alpha]/({2\eta, \eta^2\xrightarrow{\delta}\alpha}))\Rightarrow \pi_{s,w}{(\ko)^\wedge_\eta}$$
%\end{thm}

\begin{thm}
\label{thm:koconverge}
	The $C_2$-effective spectral sequence of $\kokq$ has the following form, and it converges strongly in the sense of \cite{boardman1999convergence} to $\pi_{s,w}{(\kokq)^\wedge_\eta}$:
	$$\pi_{s,w}(\HZ[\eta,\sqrt\alpha]/({2\eta, \eta^2\xrightarrow{\delta}\sqrt\alpha}))\Rightarrow \pi_{s,w}{(\kokq)^\wedge_\eta}$$
\end{thm}
\begin{pf}
	It is shown in later sections that the spectral sequence collapses at $E_2$. By the general theory of spectral sequence convergence\cite[Section 12, 6.21]{goerss2009simplicial}, the spectral sequence converges to the homotopy groups of the homotopy limit of the tower.
\end{pf}

\begin{rmk}
	By Proposition \ref{prop:redkqiskoceta}, this spectral sequence also computes the RO($G$)-graded homotopy groups of the $\eta$-completion of $\ko$.
\end{rmk}

% This is a example of the general theory about spectral sequence associated to a fibration tower.
% By a corollary in \cite[VII, Corollary 6.21]{SHT}

% \begin{cor}
% Suppose that for each integer $n \geq 0$ and each integer $s$, there
% is an integer $N$ so that
% $E^{s,s+n}_\infty\simeq E^{s,s+n}_N.$
% Then the spectral sequence converges completely.
% \end{cor}

% In our case, the spectral sequence collapse at $E_2$-page. By the corollary, it converges completely. The definition of converging completely is in \cite[VII, Definition 6.18]{SHT} and it contains two parts: first, the spectral sequence compute the limit of homotopy groups of each term in the tower; secondly, the limit of the homotopy groups is the same as the homotopy groups of the limit of the tower, i.e. the spectrum $ko_{C_2}$ (up to a completion).

% It can be shown that the spectral sequence converges to 
% $$\lim \pi_{*,*}\ReB(X_n)=\pi_{*,*}\lim \ReB(X_n)$$

% The question is that if 
% $$\pi_{*,*}\lim \ReB(X_n)= \pi_{*,*} \ReB \lim (X_n) .$$

% Realization functor preserves colimits and finite limits.

% By our computation $\pi_{*,*}\ReB(X_n)$  is has finitely many non zero term for each fixed bi-index. Therefore, this is a finite limit. 

\section{Equivariant cohomology of a point}
\label{EM}

Similar to the motivic effective slice spectral sequence, the $C_2$-effective spectral sequence makes sense integrally and no completion is needed. However, in the case of $\kokq$, it turns out that the odd primary behavior of $\kokq$ is essentially the same as for ordinary $\mathbf{ko}$.
Therefore in this paper, we exclude the odd primes by completing at $2$. The $2$-completed spectral sequence has the following form:
	$$\pi_{*,*}(\HZcplt[\eta,\sqrt\alpha]/({2\eta, \eta^2 \xrightarrow{\delta}\sqrt\alpha}))\Rightarrow \pi_{*,*}{(\kokq)^\wedge_{2}},$$
	where $\eta$ now represents a wedge summand of $\Sigma^{1,1}\HZt$ and $\sqrt\alpha$ a wedge summand of $\Sigma^{4,2}\HZcplt$. By Corollary \ref{prop:redkqiskoc2}, the spectral sequence computes the homotopy ring of $(\ko)^\wedge_2.$

	\begin{rmk}
	\label{rmk:2etacplt}
	The readers may expect to see both $2$ and $\eta$-completions in what the spectral sequence converges to. In fact, as discussed in Remark \ref{rmk:2etacpltfirst}, the $2$-completed spectrum $\kokq$ is already $\eta$-complete.
	\end{rmk}

To discuss the $E_1$-page of the $C_2$-effective spectral sequence for $\kokq$, we need to consider the homotopy rings of the Eilenberg MacLane spectra $\HZt$ and $\HZcplt$, which we denote by $(\HZcplt)_{*,*}$ and $(\HZt)_{*,*}$. 

\subsection{The homotopy ring of \texorpdfstring{$\mathbf{H}\uline{\bbF_2}$}{hzt}}\hfill\\
\label{EM1}

In this section, we describe $(\HZt)_{*,*}$. The following is a
reinterpretation of the results from \cite[Proposition 6.2]{HK}.\\

The homotopy ring $(\HZt)_{*,*}$ equals\\

$$ (\HZt)_{s,w}=
\begin{cases}
    \bbF_2, & s \leq 0 \text{ and } w \leq s \\
    \bbF_2, & s \geq 0 \text{ and } w \geq s + 2\\
    $0$, & \text{ otherwise }
\end{cases}. $$

%%%%%%%%%%%%%%%%%%%%%%%%%%%%%%%%%%%%%%%%%%%%%%%%%%%%%%%%%%%%%%%%%%%%%%%%%%%%%%%%%%%%%%%%%%%%%%%%%%%%%%%%%%%%%%%%%%%%%%%%%%%%%%%%%%%%%%%%%%%%%%%%%%%%%%%%%%%%%%%%%%%%%%%%%%%%%%%%%%%%%%%%%%%%%%%%%%%%%%%%%%%%%%%%%%%%%%%%%%%%%%%%%%%%%%%%%%%%%%%%%%%%%%%%%%%%%%%%%%%%%%%%

\begin{figure}
\begin{center}
\makebox[\textwidth]{\includegraphics[trim={0cm, 0cm, 0cm, 0cm},clip,page=1,scale=0.97]{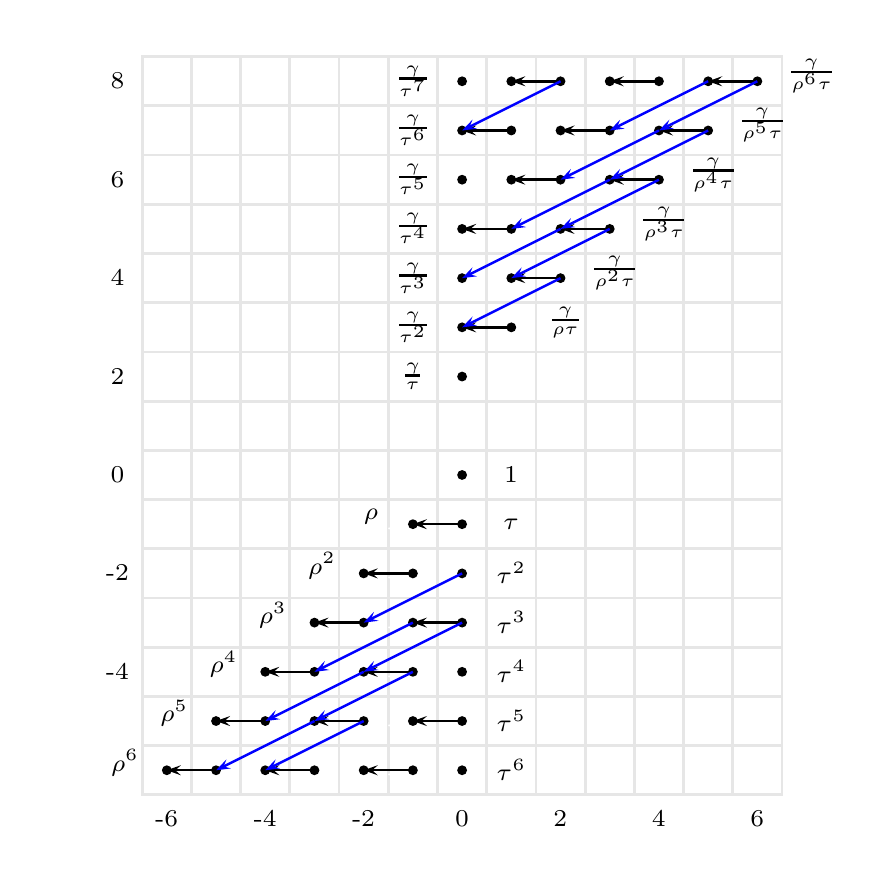}}
\caption{The homotopy ring of $\HZt.$}
\label{fig:HF}
\hfill
\end{center}
\end{figure}

Figure \ref{fig:HF} presents $(\HZt)_{s,w}$. Each dot in the figure represents a copy of $\bbF_2$. The non-zero element in degree $(0, -1)$ is called $\tau$, and the non-zero element in degree $(-1, -1)$ is called $\rho$. We refer to \cite[Section 4]{may2018structure} for geometric models for $\tau$ and $ \rho$. These two elements are polynomial generators which generate elements in the third quadrant. 

We refer to the subring $\bbF_2[\tau,\rho]\subset (\HZt)_{*,*}$ as the positive cone and the rest as the negative cone. The positive cone elements appear to have negative gradings since we adapt homological grading rather than cohomological grading.

In the negative cone, every generator has the form $\NC{i}{j}$ for some $i\geq 0,~~j\geq 1$. The element $\NC{i}{j}$ is in degree $(i,i+j+1)$. May's paper(\cite[Section 4]{may2018structure}) also gives a geometric model for $\NCt{}$ where it is named $\theta$.

The product of a positive cone element and a negative cone element is as indicated by the elements' names. When the product makes sense and formally gives an element in $(\HZt)_{*,*}$, it does. For example, we have $\NC{}{2}\cdot \tau\rho=\NCt{}$, and $\NCt{2}\cdot \tau^2=0$ since $\gamma$ is not an element in $(\HZt)_{*,*}$. 
The squares of elements in the negative cone are zero. 

%TODO multiplicative description
Certain Steenrod operations are also marked in the figure. The black horizontal arrows are $\Sq^1$ with degree $(-1,0)$. The blue arrows are $\Sq^2$ with degree $(-2,-1)$. Steenrod squares of degree three and more are not drawn in the picture. In general, the element $\Sq^n$ has degree $(-n, -\lceil{n/2}\rceil)$.

%%%%%%%%%%%%%%%%%%%%%%%%%%%%%%%%%%%%%%%%%%%%

\subsection{Bockstein spectral sequence and the homotopy ring of \texorpdfstring{$\mathbf{H}\uline{\bbZ_2}$}{hz}}\hfill\\
\label{EM2}

We consider the short exact sequence 
$$\bbZ\xrightarrow{\times 2} \bbZ\to \bbF_2.$$
It induces a cofiber sequence
$$\HZ\xrightarrow{2} \HZ \xrightarrow{\text{pr}} \HZt.$$
Passing to the homotopy rings, we obtain a $2$-Bockstein spectral sequence which converges to $(\HZcplt)_{*,*}$ and has $E_1$-page of the following form:
$$E^1_{s,f,w}=(\HZt)_{s,w}[t].$$

By the construction of the spectral sequence, the differential is the composite 
$$\HZt\xrightarrow{\delta} \Sigma^{1,0}\HZ \xrightarrow{\Sigma^{1,0}\text{pr}} \Sigma^{1,0}\HZt$$
where $\text{pr}$ is the projection map, and $\delta$ is the connecting homomorphism. The connecting homomorphism $\delta$ maps $\tau$ to $\rho$ and maps $\rho$ to zero. Passing to homotopy, we have:

\begin{prop}
\label{2Bdiff}
	Let $k$ be a positive integer. In the above 2-Bockstein spectral sequence, there are non-trivial $d_1$ differentials:
	\begin{enumerate}
		\item $d_1(\tau^{2k-1})=t\rho\tau^{2k-2}$.
		\item $d_1\left(\NC{}{2k-1}\right)=t\NCt{2k}$.
	\end{enumerate}
\end{prop}

\begin{rmk}
	By inspection, every non-zero 2-Bockstein differential is related by $\rho$ multiplications to the differentials in Proposition \ref{2Bdiff}. For example, we have $d_1(\rho^k \tau) = t \rho^{k+1}$ and $d_1\left(\frac{\gamma}{\rho^k \tau}\right) = t \frac{\gamma}{\rho^{k-1} \tau^2}$.
\end{rmk}

For degree reasons, the spectral sequence collapses after the $E_1$-page. As a result, we get the homotopy ring ${(\HZcplt)}_{*,*}$:

$${(\HZcplt)}_{s,w}=
\begin{cases}
    \bbZ_2, & s=0  \text{ and } w=2k, k\in \bbZ\\
    \bbF_2, & s \leq -1 \text{ and } w = s-2k, k\geq 0\\
    \bbF_2, & s \geq 0 \text{ and } w = s+2k+3, k\geq 0 \\
    $0$, & \text{ otherwise }
\end{cases}.$$

\begin{figure}
\begin{center}
\makebox[\textwidth]{\includegraphics[trim={0cm, 0cm, 0cm, 0cm},clip,page=2,scale=0.97]{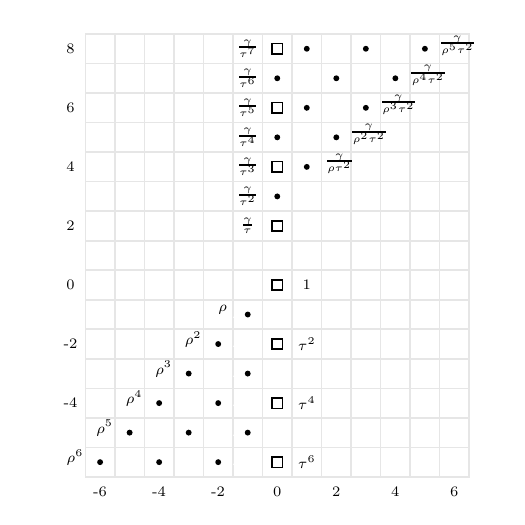}}
\caption{The coefficient ring of $\HZcplt.$}
\label{fig:HZ}
\hfill
\end{center}
\end{figure}

Figure \ref{fig:HZ} presents ${(\HZcplt)}_{s,w}$. A square in the figure represents $\bbZ_2$ while a dot represents $\bbF_2.$ We abuse notation and use the same names for elements in both $(\HZcplt)_{*,*}$ and the $E_\infty$-page of the 2 Bockstein spectral sequence. Note that the elements $\tau^{2n}$ and $\NCt{2n-1}$ detect multiple elements. Therefore in fact there are different choices of homotopy elements to be called $\tau^{2n}$ or $\NCt{2n+1}$. Definition \ref{HZname} gives specific choices.

\begin{definition}\mbox{}

	\begin{enumerate}
		\item Let $1\in(\HZcplt)_{0,0}$ be the multiplicative identity.
		\item Let $\tau^2$ be the element in $(\HZcplt)_{0,2}$ that has non-equivariant underlying image $1$ by the forgetful functor (\cite{araki1978tau}). Let $\tau^{2n}\in(\HZcplt)_{0,2n}$ be $(\tau^2)^n$.
		\item Let $\NCt{}\in(\HZcplt)_{0,-2}$ be the Hurewicz image of the quotient map (\cite{may2018structure}) $$\tilde{\theta}:S^{2,2}\to S^{2,0}\simeq S^{2,2}/C_2.$$ Let $\NCt{2n+1}$ be the element in $(\HZcplt)_{0,2}$ such that $$\NCt{2n+1}\cdot \tau^{2n}=\NCt{}.$$
	\end{enumerate}
	\label{HZname}
\end{definition}

Part of the ring structure in $(\HZcplt)_{*,*}$ is not seen from the $2$-Bockstein spectral sequence $E_\infty$-page. We discuss these multiplicative relations below. 

%We take $1$ to be the multiplicative identity. We take $\tau^2$ to be the class in $(\HZcplt)_{*,*}$ that has underlying image $1$(\cite{araki1978tau}). And we take $\NCt{}$ to be the element that is the Hurewicz image of the quotient map $$\tilde{\theta}:S^{2,2}\to S^{2,0}\simeq S^{2,2}/C_2.$$

%\begin{prop}
%The Hurewicz image of the quotient map:
%$$\tilde{\theta}:S^{2,2}\to S^{2,0}\simeq S^{2,2}/C_2.$$
%is detected by $\NCt{}$.
%\end{prop}

%\begin{proof}
%Let $\bbZ_2 \langle S^{p,q}\rangle$ denote the configurations of points on $S^{p,q}$ with labels in $\bbZ_2$. It is the usual Dold-Thom model for $(p,q)$-th Eilenberg-Maclane spaces\cite{dos2003note}. Consider the following diagram:
%$$\begin{tikzcd}
%S^{2,2}\ar[rr,"\theta"]\ar[rd,"\tilde{\theta}"] & & \bbZ_2\langle S^{p,q}\rangle \\	
%&S^{2,0}\ar[ru,"\iota"] & 
%\end{tikzcd},$$
%where $\iota$ is the canonical map sending each point $x$ to the configuration $[x]$.
%	The proof of \cite[Proposition 4.5]{may2018structure} shows that the element $\theta\in \HZ_{0,-2}$ is the Hurewicz image of $\tilde{\theta}$. We take the inclusion map $\bbZ\to \bbZ_2$ and the induced map on coefficient rings gives the desired element.
%\end{proof}

\begin{prop}
In $(\HZcplt)_{*,*}$, there are multiplicative relations:
\label{multiinHZ}
\begin{enumerate}
	\item $$\left(\frac{\gamma}{\tau}\right)^n=2^n\frac{\gamma}{\tau^{2n-1}}.$$
	\item $$\tau^2\cdot \frac{\gamma}{\tau}=2.$$
\end{enumerate}
\end{prop}

\begin{proof}
	Since the non-equivariant underlying image of the quotient map $\tilde{\theta}$ is of degree 2, the element $\NCt{}$ has underlying image $2$. Since $\tau^2$ has underlying image $1$, we obtain that $\NCt{2n-1}$ has underlying image $2$ for any positive integer $n$. We prove the results by comparing the non-equivariant underlying maps.
	\begin{enumerate}
		\item The generator $\NCt{2n-1}$ of $(\HZcplt)_{0,-2}$ has underlying image $2$, while the element $\left(\frac{\gamma}{\tau}\right)^n$ has underlying image $2^n$.  
		\item The product lands in degree $(0,0)$ and has underlying image $2.$
		\end{enumerate}
%	

%	The elements $\tau$ and $\NCt{}$ are in the Hurewicz image of homotopy groups of $C_2$-equivariant sphere. The element $\NCt{}$ can be realized as as $\alpha: S^{2,2}\to S^{2,0}\in \pi_{0,-2}^{C_2}(S^{0,0})$. Its underlying image in homotopy groups of sphere is $2\in \pi_{0}(S^0)$. And the underlying image of $\tau^2$ is $2\in \pi_{0}(S^0)$.
%	\begin{enumerate}
%	\item Therefore $\alpha^n$ has underlying image $2^n\in \pi_{0}(S^0)$. Since $\alpha^n$ is in degree $(0,-2n)$ and $\tau$ has underlying image $1$, we have that 
%	$$\left(\frac{\gamma}{\tau}\right)^n=2^n\frac{\gamma}{\tau^{2n-1}}.$$
%	\item 

\end{proof}

%
%
%\subsection{Cohomology of \texorpdfstring{$H\uline{\mathbb{F}_{2^n}}$}{hz}}\hfill\\
%
%\mytodo{chart1}

%

%
%
%
\section{Computation of the \texorpdfstring{$C_2$}{C2}-effective spectral sequence}
\label{diff}

\subsection{The \texorpdfstring{$C_2$}{C2}-effective \texorpdfstring{$E_1$}{E1}-page}\hfill\\
\label{E1page}

%of the equivariant effective slice spectral sequence
%\subsection{\texorpdfstring{$E_1$}{E1} page by coweight}\hfill\\

%We compute the $2$-complete effective spectral sequence of $\ko$ which converges to $\pi_{*,*}^{C_2}\ko^\wedge_{\eta}.$ 

Using the results from Sections \ref{EM1} and \ref{EM2}, we have a description of the $E_1$-page of the $C_2$-effective spectral sequence of $\kokq.$ The spectral sequence is tri-graded. We denote the gradings by $(s,q,w)$, where $s$ is the stem, $q$ is the slice filtration, and $w$ is the weight. 

\begin{thm}
	Additively, the $E_1$-page of the effective spectral sequence of $\kokq$ is the following:
	$$E^1_{*,*,*}=(\HZcplt)_{*,*}[v_1^2]\oplus (\HZt)_{*,*}[h_1,v_1^2]\{h_1\}$$
	where $\abs{h_1}=(1,1,1)$ and $\abs{v_1^2}=(4,2,2).$ 
\end{thm}

\begin{rmk}
We give the precise meanings of the elements $h_1$ and $v_1^2.$	
\begin{enumerate}
	\item By (\ref{koc2e1spec}), we have that $s_1(\kokq)\cong \Sigma^{1,1}\HZt$. The element $h_1$ represents the element in $ \pi_{*,*}(s_1(\kokq))\cong (\HZt)_{*,*}[1,1]$ that is the $(1,1)$ suspension of the unit element $1\in (\HZt)_{*,*}$.
		\item By (\ref{koc2e1spec}), we have that $s_2(\kokq)\cong \Sigma^{4,2}\HZcplt\vee \Sigma^{2,2}\HZt.$ As a result, we can express its homotopy groups as a direct sum: $\pi_{*,*}(s_2(\kokq))\cong (\HZcplt)_{*,*}[4,2]\oplus (\HZt)_{*,*}[2,2].$ The element $v_1^2$ represents the element in the summand $(\HZcplt)_{*,*}[4,2]$ that is the $(4,2)$ suspension of the unit element $1\in (\HZcplt)_{*,*}$.
\end{enumerate}
\end{rmk}
%TODO broken line

The multiplicative relations are endowed from the ring structures in $(\HZcplt)_{*,*}$ and $(\HZt)_{*,*}$, together with the extra relation $\delta:\eta^2 \to \sqrt\alpha.$ Since the connecting homomorphism $\delta$ maps $a\in (\HZt)_{*,*}$ to $\Sq^1(a)\in (\HZcplt)_{*,*},$ this extra relation translates to the following formula after passing to homotopy:
\begin{equation}
\label{multirelationhomo}
	ah_1\cdot bh_1=ab h_1^2 +\Sq^1(a)\Sq^1(b) v_1^2, ~~a, b \in (\HZt)_{*,*}.
\end{equation}

%The term $\Sq^1(a)\Sq^1(b) v_1^2$ is from the multiplicative relation $\delta:\eta^2\to \alpha.$ After passing to homotopy, the connecting homomorphism $\delta$ maps $a$ to $\Sq^1(a).$
After multiplication by $v^{2}_1$, the relation (\ref{multirelationhomo}) generates other relations for products $ah_1v_1^{2m}\cdot bh_1v_1^{2n}$.
%For example, we have
%$$ah_1v_1^2\cdot bh_1v_1^4=ab h_1^2v_1^6 +\Sq^1(a)\Sq^1(b) v_1^8, ~~a, b \in (\HZt)_{*,*}.$$
\begin{eg}
\label{example:tauhsquare}
	It's worth mentioning that $(\tau h_1)^2$ and $\tau^2 h_1^2$ are two different elements because of this extra multiplicative relation. When $a=b=\tau$, we have that 
	$$(\tau h_1)\cdot (\tau h_1)=\tau^2\cdot h_1^2+\rho^2 \cdot v_1^2.$$
	This relation leads to interesting results in the $C_2$-effective $d_1$ differential computation later in Section \ref{subsec:d1diff}.
\end{eg}

Below is a table which lists the multiplicative generators and their degrees in the $E_1$-page. We call an element in $E_1$-page a positive (resp., negative) cone element if it is a suspension of an element in the positive (resp., negative) cone of $(\HZt)_{*,*}$ or $(\HZcplt)_{*,*}.$

Charts of the $E_1$-page by coweights are attached in Section \ref{charts}. 

\begin{table}[h]
    \centering
    \begin{tabular}{ccc}
    Elements & $(s,f,w)$ & $s-w$\\
    \hline
    $\rho$ & (-1,0,-1) & 0\\
    $\tau^2$ & (0,0,-2) & 1\\
    $h_1$  & (1,1,1) & 0\\
    $\tau h_1$ & (1,1,0) & 1\\
    $v_1^2$ & (4,2,2) & 2\\
    $\NC{i}{j}$ & $(i, 0 ,i+j+1)$ & $-j-1$
\end{tabular}
    \caption{Symbols for multiplicative generators in the $C_2$-effective $E_1$-page}
    \label{E1}
\end{table}

\subsection{The \texorpdfstring{$d_1$}{d1} differentials}\hfill\\
\label{subsec:d1diff}

We discuss $d_1$ differentials supported by the positive cone elements and the negative cone elements separately.

\begin{prop}
\label{PCd1}
The nontrivial $d_1$ differentials supported by the positive cone elements are determined by the following, together with the Leibniz rule:
\begin{enumerate}
	\item $d_1(\tau^2)=\rho^2 \cdot \tau h_1.$
	\item  $d_1(v_1^2)=\tau h_1\cdot h_1^2.$	
\end{enumerate}
\label{PCdiff}
\end{prop}

\begin{pf}
R\"ondigs and \O stvær computed the first slice differentials on the spectrum level. The maps $\delta_1(q):s_1\kq \to \Sigma^{1,0}s_{q+1}\kq$ are described in terms of Steenrod squares in \cite[Theorem 5.5]{rondigs2016slices}. After applying the realization functor, we get the differential maps in the $C_2$-effective slice tower in terms of $C_2$-equivariant Steenrod squares. The result follows by computing Steenrod squares on the homotopy elements. %On the positive cone multiplicative generators:
%\begin{align*}
%	d_1(\tau^2)&=\Sq^2(\tau^2)h_1=\rho \tau^2 h_1,\\
%	d_1(\rho)&=\Sq^2(\rho)h_1=0,\\
%	d_1(v_1^2)&=\Sq^2(1)h_1v^2_1+\tau h_1^3=\tau h_1^3,\\
%	d_1(h_1)&=\Sq^3\Sq^1(1)h_1v_1^2+\Sq^2(1)h_1^2=0 \\
%	d_1(\tau h_1)&=\Sq^3\Sq^1(\tau)v_1^2+\Sq^2(\tau)h_1=0
%\end{align*}
%The positive cone are generated by the above multiplicative generators. Therefore, they generate all other differentials in the positive cone under the Leibniz rule.
\end{pf}

\begin{rmk}
%We give two different proof here. One is from motivic effective slice spectral sequence and on the spectrum level. The other is from motivic Adams spectral sequence and on the algebraic level.
There are different ways to analyze the $d_1$ differentials in the $C_2$-effective spectral sequence. We can push forward or pull back what is known in $\bbR$-motivic homotopy theory and classical homotopy theory using the top horizontal and right vertical functors in Diagram \ref{magicsquare}. 

For example, we can compare it to the classical Adams-Novikov spectral sequence for the real $K$-theory spectrum $\mathbf{ko}.$ By \cite{levine2015ANSS}, Betti realization induces an isomorphism between the $\bbC$-motivic effective slice spectral sequence and the Adams-Novikov spectral sequence up to reindexing. In the Adams-Novikov spectral sequence of $\mathbf{ko}$, the classical element $v_1^2$ hits $h_1^3$ by a $d_3$ differential. Thus by inspection, we conclude that in the $\bbC$-motivic effective slice spectral sequence of the $\bbC$-motivic spectrum $\kq$, the element $v_1^2$ supports a $d_1$ differential which kills $\tau h_1\cdot h_1^2$. The same is true in the $\bbR$-motivic stable homotopy category. We then take the Betti realization to get the $C_2$-equivariant effective differential $d_1(v_1^2)=\tau h_1\cdot h_1^2$. 
% see edits 3
%Another example is by comparing it to the $\bbR$ motivic Adams spectral sequence. The work of \cite{dugger2013motivic} computes the $\bbR$ motivic Adams spectral sequence of the motivic sphere. The motivic Adams $d_2$ differential $d_2(\tau^2)=\rho^2\cdot \tau h_1$ gives the motivic effective slice $d_1$ differential which realizes to the $C_2$-effective $d_1$ differential.
\end{rmk}

As in the proof of Proposition \ref{PCd1}, the spectrum level $d_1$ differential maps in the $C_2$-effective slice tower are known. By passing to homotopy groups, we can get formulas for differentials in both cones by computing the Steenrod squares. 
%However, Steenrod operations are harder to describe in the negative cones.
Below we adapt another more algebraic approach to compute the differentials on negative cone elements.

\begin{prop}
For any non-negative integer $k$, 
there are $d_1$ differentials
$$d_1 \left(\NC{2}{4k+2}\right) = \NCt{4k+3} h_1.$$ 
\label{NCdiff}
\end{prop}

%\begin{prop}
%Let $m,k$ be two positive integers. The nontrivial $d_1$ differentials in the negative cone are determined by the following differentials:
%\begin{enumerate}
%	\item $$d_1\left(\NC{i}{j}\right)=
%	\begin{cases}
%		\NC{i-2}{j+1} h_1, & j=4k+2\\
%		0, & \text{otherwise}
%	\end{cases}.$$
%%	\item \NC{i-2}{j+1} h_1,$$ when $j=4k+2.$
%	\item $$d_1\left(\NC{i}{j} h_1^m\right)=
%	\begin{cases}
%		\NC{i-2}{j+1} h_1^{m+1}+ \NC{i-4}{j+3} v_1^2h_1^m, & j=4k+1\\
%		\NC{i-2}{j+1} h_1^{m+1}, & j=4k+2\\
%		0, & \text{otherwise}
%	\end{cases}.$$
%	\item $$d_1 \left(\NC{i}{j} v_1^2\right)=
%	\begin{cases}
%		\NC{i-2}{j+1} v_1^2h_1+ \NC{i}{j-1} h_1^3, & j=4k+2\\
%		\NC{i}{j-1} h_1^{m+3}+ \NC{i-2}{j+1} v_1^2h_1^{m+1}, & \text{otherwise}
%	\end{cases}.$$
%	\item $$d_1\left(\NC{i}{j} v_1^2h_1^m \right)=
%	\begin{cases}
%		\NC{i}{j-1} h_1^{m+3}+ \NC{i-4}{j+3} v_1^4h_1^m, & j=4k+1\\
%		\NC{i}{j-1} h_1^{m+3}, & j=4k\\
%		\NC{i}{j-1} h_1^{m+3}+ \NC{i-2}{j+1} v_1^2h_1^{m+1}, & \text{otherwise}
%	\end{cases}.$$
%\end{enumerate}
%\end{prop}
\begin{pf}

The Leibniz rule gives that:
$$
 	0= d_1\left(\NC{2}{4k+2}\cdot \tau^{4k+2}\right)=d_1\left(\NC{2}{4k+2}\right)\cdot \tau^{4k+2}+\NC{2}{4k+2}\cdot d_1(\tau^{4k+2}).
$$
The result follows by Proposition \ref{PCd1} which gives that $d_1(\tau^{4k+2})=\rho^2\tau^{4k+1}h_1$. 
% $$d_1\left(\NC{i}{j}\right)=
% \begin{cases}
%	\NC{i-2}{j+1} h_1, &j=4k+2\\
%	0, & j=4k
%\end{cases} .$$
%\item The formula can be computed using Leibniz rule and the multiplicative relation $$ah_1\cdot bh_1=ab h_1^2 +\Sq^1(a)\Sq^1(b) v_1^2.$$ Same for (3) and (4).
%
%\end{enumerate}
\end{pf}

\begin{rmk}
In fact, we can prove that 
\begin{equation}
\label{equa:NCd1}
	d_1\left(\NC{i}{j}\right)=		
	\begin{cases}
		\NC{i-2}{j+1} h_1, & j=4k+2,\\
		0, & \text{otherwise.}
	\end{cases}
\end{equation}	
\end{rmk}
When $j$ is odd, this follows by inspection. When $j$ is even, this follows by the Leibniz rule and Proposition \ref{PCd1}.

\begin{rmk} 
The additive generators in the negative cone in the $E_1$-page are decomposable in one of the following forms:
 	$$\NC{i}{j} \cdot h_1^p v_1^{2q} \text{ or } \NC{i}{j} \cdot \left(\tau h_1\right)\cdot h_1^p v_1^{2q}.$$ Therefore, we can use Propositions \ref{PCdiff} and \ref{NCdiff} together with the Leibniz rule to determine all other $d_1$ differentials. We include the following example where the multiplicative relation (\ref{multirelationhomo}) is involved.
\end{rmk}

\begin{eg}
	We compute $d_1\left(\NC{i}{4k+1} h_1 \right)$ where $k$ is a non-negative integer.  
	
	The element $\NC{i}{4k+1} h_1$ is decomposable as 
	$$\NC{i}{4k+1} h_1= \NC{i}{4k+2}\cdot\tau h_1.$$
	We use Proposition \ref{NCdiff} and that $d_1(\tau h_1)=0.$ By the Leibniz rule, 
	\begin{align*}
		d_1\left(\NC{i}{4k+1} h_1 \right)&=d_1\left(\NC{i}{4k+2}\right)\cdot\tau h_1\\
		&=\NC{i-2}{4k+3} h_1 \cdot \tau h_1\\
		&=\NC{i-2}{4k+2} h_1^2 + \Sq^1\left(\NC{i-2}{4k+3}\right)\Sq^1(\tau)v_1^2\\
		&=\NC{i-2}{4k+2} h_1^{2}+ \NC{i-4}{4k+4} v_1^2.
	\end{align*}
Thus the sum is killed. Therefore, after the $E_1$-page, the element $\NC{i-2}{4k+2} h_1^{2}$ is identified with $\NC{i-4}{4k+4} v_1^2.$
\end{eg}

\subsection{The \texorpdfstring{$E_\infty$}{collapse}-page }\hfill\\

In this section, we show that the spectral sequence collapses at the $E_2$-page. 

%\begin{lemma}
%	Let $x$ and $y$ be two elements on page $E_{r}$ such that $d_r(x)=y$. If $x$ is infinitely $\rho$ divisible, then $y$ is also infinitely $\rho$ divisible.
%	\label{rholem}
%\end{lemma}
%\begin{proof}
%		Suppose $\frac{y}{\rho^{k}}=0$ for some $k\geq 0.$ We have that $y=d_r(\frac{x}{\rho^k}\cdot \rho^k)=\rho^k\cdot d_r(\frac{y}{\rho^{k}})$
%\end{proof}

\begin{prop}
	In the effective equivariant spectral sequence for $\kokq$, the only nontrivial differentials are on the $E_1$-page.
\end{prop}

\begin{pf}

For degree reasons, the element $\rho$ is a permanent cycle. Therefore there cannot be a differential from a $\rho$ torsion element to a $\rho$ torsion free element. Similarly, there cannot be a differential from an infinitely $\rho$-divisible element to an element that is not infinitely $\rho$-divisible. These rule out almost all possible higher differentials. 

The only exception is in coweight $4k+1$. The negative cone elements $\NC{}{4i+2}v_1^{4j}$ might possibly support non-trivial higher differentials. We show this cannot happen by computation. The element $v_1^{4j}$ is a permanent cycle by the discussion before. Therefore, for any $r\geq 2, $ we have that
$$d_r\left(\NC{}{4i+2}v_1^{4j}\right)=d_r\left(\NC{}{4i+2}\right)\cdot v_1^{4j}=0$$ since $d_r\left(\NC{}{4i+2}\right)=0$ for degree reasons.
%	\item  Negative cone: 
%	
%	$n=4k+1$: The only possible elements that supports non trivial differentials are those on the bottom right corner of the paralellogram, of the form $\NC{}{4i+2}v_1^{4j},$ and when $k\geq 0.$ Since $v_1^{4j}$ is a permanent cycle,
%	$$d_r\left(\NC{}{4i+2}v_1^{4j}\right)=d_r\left(\NC{}{4i+2}\right)\cdot v_1^{4j}=0.$$
\end{pf}

The charts of the $E_\infty$-page by coweights are attached in Section \ref{subsec:charts}.

% \input{4diff/collaps.tex}

% \subsection{\texorpdfstring{$E_\infty$}{Einf} page by coweight}\hfill\\

% \input{4diff/Einf.tex}

\section{The homotopy of \texorpdfstring{$(\kokq)^\wedge_2$}{koc2} }
\label{ext}
%\subsection{The elements in the homotopy}\hfill\\

In this section, we solve the extension problems in the spectral sequence to obtain the homotopy groups of $\kokq$ after $2$-completion. We follow the definition of hidden extension as in \cite{isaksen2014stable}.

\begin{definition}
 Let $A$ be an element of $\pi_{*,*}$ that is detected by an element
$a$ of the $E_\infty$-page of the spectral sequence. A hidden extension by
$A$ is a pair of elements $b$ and $c$ of the $E_\infty$-page such that:
\begin{enumerate}
  \item the product $a\cdot b$ equals zero in the $E_\infty$-page.
  \item There exists an element $B$ of $\{b\}$ such that $A\cdot B$ is contained in $\{c\}$.
  \item If there exists an element $B'$ of $\{b'\}$ such that $A\cdot B'$
is contained in $\{c\}$, then the filtration of $b'$ is less than or equal to the filtration of $b$.
\end{enumerate}
	
\end{definition}

In the spectral sequence, the elements $\tau^4$ and $v_1^4$ are permanent cycles and survive to the $E_\infty$-page. We abuse the notation and use $\tau^4$ and $v_1^4$ to denote elements in homotopy. The definitions of these two elements in homotopy are as follows:

\begin{definition}
\label{deftau4v4}
	Let $\tau^4$ be an element in $\pi^{C_2}_{0,-4}((\kokq)^\wedge_2)$ that is detected by $\tau^4$. Let $v_1^4$ be the element in $\pi^{C_2}_{8,4}((\kokq)^\wedge_2)$ that is detected by $v_1^4$ and satisfies that $v_1^4\cdot \rho^4=\eta^4 \cdot \tau^4.$
\end{definition}

\begin{rmk}
	There are different choices of $\tau^4$ because of the existence of elements in higher filtration. However these different choices make no differences for our purpose. Once we fix the choice of $\tau^4$, the choice of $v_1^4$ is unique.
\end{rmk}	

The $E_\infty$-page organized by coweight shows an almost $4$-periodic pattern: if we slide the $(k+4)$-th page down left by $(8,4)$ and truncate, we see the same pattern as in the $k$-th page. In fact, this is almost $v_1^4$-periodicity in the sense of the following theorem.

\begin{thm}
	Multiplication by $v_1^4$ gives a homomorphism on homotopy groups:
	$$\pi_{s,w}((\kokq)^\wedge_2)\to \pi_{s+8,w+4}((\kokq)^\wedge_2)$$
	which is 
	\begin{enumerate}
	  	\item injective if $s-w=-4$.
 	 	\item zero if $s-w=-5$.
 	 	\item bijective otherwise.
	\end{enumerate}
\end{thm}	
	
\begin{proof}
	The above observation is already visible from the $E_\infty$-page and does not involve hidden extensions. 
\end{proof}	

In addition to the $v_1^4$-periodicity, the effective spectral sequence sees a $\tau^4$ periodicity in the sense of the following theorem. This has been observed in work of \cite[Theorem 11.15]{GHIR}. 

\begin{thm}
\label{tauperiod}
	Multiplication by $\tau^4$ gives a homomorphism on homotopy groups:
	$$\pi_{s,w}((\kokq)^\wedge_2)\to \pi_{s,w-4}((\kokq)^\wedge_2)$$
	which is 
	\begin{enumerate}
  \item injective if $s-w=-4$.
  \item zero if $s-w=-5$.
  \item bijective otherwise.
\end{enumerate}

\end{thm}

Part of the $\tau^4$ multiplication relation is already in the multiplicative structure in $\HZcplt_{*,*}$ shown in Proposition \ref{multiinHZ}. 
%In fact, we have multiplications on $E_\infty$-page:
%
%\begin{prop}
%\label{tauext}
%On the $E_1$-page, there are $\tau^4$ multiplications:
%\begin{enumerate}
%  \item  $\NCt{}\cdot \tau^4=2\tau^2$.
%  \item  $\NC{2}{3}\cdot (\tau h_1)^2\cdot \tau^4=2v_1^2$.
%  \item  $\NCt{3}\cdot \tau^4=2$.
%\end{enumerate}
%\end{prop}
However, the rest is hidden in the $C_2$-effective $E_\infty$-page. To prove the $\tau^4$ periodicity theorem, we need to determine these hidden multiplicative relations. The following proof of Theorem \ref{tauperiod} uses propositions to be shown below in Section \ref{subsec:hiddenext}.
\begin{pf}
	The conclusion follows from Proposition \ref{tauext11}(2) in coweight $4k+1$ and from Proposition \ref{tauext11}(1) in coweight $4k+2$.
	The rest follows from Proposition \ref{multiinHZ}.
\end{pf}

Before discussing the hidden extensions, we first name some elements in homotopy. 

We consider the elements $\eta$ and $\rho$ in the homotopy of the $C_2$-equivariant sphere spectrum. The first Hopf map $\eta$ is in degree $(1,1)$. The element $\rho$ is in degree $(-1,-1)$ and is a lift of the element $\rho\in \HZcplt$ (see \cite{may2018structure}). Note that an element with the same name $\rho$ also exists in $\bbR$-motivic homotopy theory (see \cite{voevodsky2003mz}), and it realizes to the $C_2$-equivariant $\rho$.
 We abuse notation and use $\rho$ and $\eta$ to denote the images of these two elements under the map induced by the unit map $S^{0,0}\to \cpltt{\kokq}$. By definition, the $E_1$-page element $\rho$ detects the homotopy element $\rho.$
 %\begin{definition}
%	Let $\rho\in \pi^{C_2}_{-1,-1}(\ko)$ be the Hurewicz image of the inclusion $S^{0,0}\hookrightarrow S^{1,1}$ of fixed points. Let $\eta\in \pi_{1,1}^{C_2}(\ko)$ be the Hurewicz image of the equivariant Hopf map $\eta\in\pi_{1,1}^{C_2}(S^{0,0}).$
%\end{definition}

\begin{prop}
	The element $\eta\in \pi_{1,1}^{C_2}(\cpltt{\kokq})$ is detected by $h_1.$ 
\end{prop}
\begin{proof}
  We show the $\bbC$-motivic analogue is true.
  
	In $\bbC$-motivic homotopy theory, work of \cite{levine2015ANSS} shows that the $\bbC$-motivic effective slice spectral sequence realizes to the classical Adams-Novikov spectral sequence after re-indexing. As a consequence, the $\bbC$-motivic $\eta$ is detected by $h_1$ in the $\bbC$-motivic effective slice spectral sequence. By Diagram (\ref{magicsquare}), the same result holds in $\bbR$-motivic homotopy theory and $C_2$-equivariant homotopy theory.
\end{proof}
%\begin{prop}
%	The element $\rho\in \pi_{-1,-1}^{C_2}(\kocplt)$ is detected by $\rho$.
%\end{prop}
%
%\begin{proof}
%We prove by showing the $\bbR$-motivic analogue is true. The result will follow by realizing to $C_2$-equivariant.
%
%In $\bbR$0motivic homotopy category, the unit map $S^{0,0}\to \KQ$ factors through $\kq$ by the fact that the sphere spectrum is very effective by definition. By \cite[Lemma 2.28]{rondigs2019motivicsphere}, the unit map $S^{0,0}\to \kq$ induces an isomorphism on the zeroth effective slice which is equivalent to $\MZ$. In other words, we have 
%$$\pi_{-1,-1}S^{0,0}\to \pi_{-1,-1}(\kq^\wedge_{\eta})\to \pi_{-1,-1}\MZ\ni \rho,$$
%where the first map is induced by the unit map and the second map is induced by the map to the zeroth cocover, or equivalently the zeroth slice by effectiveness of $\kq$ and $S^{0,0}$. Therefore, the $E_1$-page element $\rho$ lifts and is sent by the unit map to the homotopy of $\kq^\wedge_{\eta}$. The $\bbR$-motivic results follows by the definitions of the elements $\rho$.
%	 \end{proof}

\subsection{Hidden extensions and periodicity}\hfill
\label{subsec:hiddenext}

In this section, we analyze the relevant hidden extensions needed for Theorem \ref{tauperiod}. Most of the arguments in the proofs of this section are easier to see when accompanied by charts in Section \ref{subsec:charts}. We suggest readers refer to the charts while reading the proofs in this subsection.

First we analyze the hidden $\rho$ extensions. The following well-known lemma is useful when analying $\rho$ extensions. One proof can be found in \cite{GHIR}.

\begin{lemma}
\label{rholem}
Let $X$ be a $C_2$-equivariant spectrum, and let $\chi$ belong to
$\pi_{n,k}(X)$. Let $U$ denote the forgetful functor from $\SH_{C_2}$ to $\SH.$ The element $\chi$ is divisible by $\rho$ if and only if its underlying class $U(\chi)$ in
$\pi_{n}(U(X))$ is zero.
\end{lemma}
%\begin{proof}
%	The equivariant isotropy separation sequence: 
%	$${C_2}_+\to S^{0,0}\xrightarrow{\rho} S^{1,1}$$
%	induces a long exact sequence:
%	$$\cdots\to \pi_{n+1,k+1}^{C_2}X\xrightarrow{\rho}\pi_{n,k}^{C_2}(X)\xrightarrow{\iota^*} \pi_n(\iota^*)\to \pi_{n+2,k+1}(X)\xrightarrow{\rho}\cdots.$$
%	
%	Therefore, an element $c$ is divisible by $\rho$ if and only if its underlying class $\iota^*{c}$ in
%$\pi_{n}(\iota^*X)$ is zero.
%
%\end{proof}
%

As a corollary of the lemma, we have the following:
\begin{cor} 
\label{rhocor}
There are hidden $\rho$ extensions
\begin{enumerate}
  \item  from $\NCt{} v_1^2$ to $h_1^3.$ 
%  \item and after multiplying $h_1$ or $v_1^4$ powers.
  \item  from $2 v_1^2 $ to $ (\tau h_1)^2\cdot h_1.$
%  \item  and after multiplying $v_1^4$ powers.
\end{enumerate}

\end{cor}

\begin{pf}
The elements $h_1^3$ and $(\tau h_1)^2\cdot h_1$ detect homotopy elements whose underlying images are in  $\pi_3(\cpltt{\mathbf{ko}})=0.$
By Lemma \ref{rholem}, these homotopy elements are divisible by $\rho.$ The results follow by inspection (see the coweight $0$ chart on Page \pageref{fig:sample} for $(1)$, and the coweight $2$ PC chart on Page \pageref{fig:E22} for $(2)$).
\end{pf}

\begin{rmk}
\label{multirmk}
These two hidden $\rho$ extensions generate other hidden $\rho$ extensions after $\tau^4$ and $v_1^4$ multiplications (e.g., from $2\tau^2 v_1^2$ to $\tau^4 h_1^3$ and from $2 v_1^6 $ to $ (\tau h_1)^2\cdot h_1v_1^4$). In addition, the first one generates more $\rho$ extensions after $h_1$ multiplication (e.g., from $\NCt{} h_1v_1^2$ to $h_1^4$).
\end{rmk}

\begin{definition}
\label{alphadef}
Let $\alpha$ be the element in $\pi^{C_2}_{4,4}(\cpltt{\kokq})$ detected by $\NCt{}v_1^2$ such that $\rho\cdot\alpha=\eta^3$. 
%  \item Let $\tau^4\alpha$ be the element in $\pi^{C_2}_{8,4}(\ko)$ such that $\rho\cdot\tau^4\alpha=\tau^4\eta^3$. The element $\tau^4\alpha$ is detected by $2\tau^2 v_1^2.$ We will verify the name in later subsections.	
\end{definition}

\begin{rmk}\mbox{}

\label{tauext1}
\begin{enumerate}
	\item There is a unique choice of $\alpha$ because of the relation that $\rho \cdot \alpha =\eta^3$. There are no possible error terms in higher filtrations.
	\item The element $\tau^4\alpha$ is detected by $\NCt{}v_1^2\cdot \tau^4=2\tau^2 v_1^2$ (See Proposition \ref{multiinHZ}(2)).
\end{enumerate}
\end{rmk}

\begin{prop}
\label{hext}
There are hidden $\eta$ extensions:
\begin{enumerate}
  \item from $2\tau^{2} $ to $ \rho(\tau h_1)^2.$
  \item from $2\tau^{2} v_1^2$ to $ \rho^3 v_1^{4}.$
\end{enumerate}
\end{prop}

\begin{proof}\mbox{}
The results follow by comparing the positive cone and the negative cone elements. The relevant charts are in coweight $2$ (Page \pageref{fig:E22}), and in coweight $0$ (Page \pageref{fig:E24}, \pageref{fig:sample}).
\begin{enumerate}
	\item  
%	By Proposition \ref{multiinHZ}, we have $2\tau^2=\NCt{}\cdot \tau^4$ and $2v_1^2=\NC{2}{3}(\tau h_1)^2\cdot \tau^4$.
%	Then the element $2\tau^2$ detects $a\tau^4$ and $2v_1^2$ detects $b\tau^4$.

	We use $\mu$ to denote the homotopy element detected by $\NCt{}$ and use $\nu$ to denote the homotopy element detected by $\NC{2}{3}(\tau h_1)^2$. The multiplicative relation (\ref{multirelationhomo}) is involved here. Note that $\NC{2}{3}(\tau h_1)^2$ equals $\NC{2}{}h_1^2+\NCt{3}v_1^2$ (see Example \ref{example:tauhsquare}).
	
		We have the relation in the $E_\infty$-page that $$\rho^2\cdot \NC{2}{3}(\tau h_1)^2= h_1^2\cdot \NCt{}.$$ 
	Therefore in homotopy, we have $\rho^2 \nu = \eta^2 \mu .$ After multiplication by $\tau^4$, we have $$\rho^2 \nu\tau^4=\eta^2 \mu \tau^4.$$
	
	By Proposition \ref{multiinHZ}, the element $\nu  \tau^4$ is detected by $\NC{2}{3}(\tau h_1)^2\cdot\tau^4=2v^2_1.$ Therefore, 
	the element $\rho^2  \nu  \tau^4$ is detected by $(\tau h_1)^2 h_1\cdot \rho$ by the hidden $\rho$-extension in Proposition \ref{rhocor}.
	
	On the other hand, the element $\mu \tau^4$ is detected by $\NCt{}\cdot \tau^4=2\tau^2.$ The result follows.
	\item We have the following relation in homotopy by the definitions of $\alpha$ and $v_1^4$:
  	$$\eta \tau^4 \alpha \rho=\tau^4 \eta^4=v_1^4 \rho^4.$$ 
  	Therefore, we have that $\eta\tau^4 \alpha$ is detected by $\rho^3   v_1^4.$
  	By Remark \ref{tauext1} (2), the element $\tau^4\alpha$ is detected by $2\tau^2v_1^2.$ The result follows.
\end{enumerate}	
\end{proof}

\begin{prop}
\label{tauext11}
There are hidden $\tau^4$ extensions 
\begin{enumerate}
	\item from $\NC{}{}h_1$ to $(\tau h_1)^2$.
	\item from $\NC{}{2}$ to $\tau h_1$.
\end{enumerate}

\end{prop}

\begin{pf}\mbox{}
The relevant charts are in coweight $2$ (Page \pageref{fig:E22}) and coweight $1$ (Page \pageref{fig:E21}).
\begin{enumerate}
\item By Proposition \ref{hext} (1), there is an $\eta$ extension from $\tau^4 \cdot \NCt{}=2\tau^2$ to $\rho(\tau h_1)^2$. Since we have $\NC{}{}h_1 \cdot \rho=\NCt{} \cdot h_1$, we get a $\tau^4$ extension from $\NC{}{}h_1\cdot \rho$ to $\rho(\tau h_1)^2$. The result follows by dividing by $\rho$ and inspection. 
%\item since there is no $\rho$-torsion.

\item Since $\NC{}{2}\cdot \tau h_1=\NC{}{}h_1,$ the result follows by the $\tau^4$ extension in (1).
\end{enumerate}	
\end{pf}

\begin{rmk}
Similarly to Remark \ref{multirmk}, we obtain more $\eta$ extensions by multiplying the extensions in Proposition \ref{hext} by powers of $\tau^4$ or $v_1^4$, and more $\tau^4$ extensions by multiplying the extensions in Proposition \ref{tauext11} by powers of $v_1^4$.
\end{rmk}

Propositions \ref{tauext11} and \ref{multiinHZ} together prove Theorem \ref{tauperiod}.\\
%\begin{pf}
%	The conclusion follows from \ref{tauext11}(2) in coweight $4k+1$ and from \ref{tauext11}(1) in coweight $4k+2$.
%	The rest follows from Proposition \ref{multiinHZ}.
%\end{pf}

	Now we have shown two periodicities. The $v_1^4$ and $\tau^4$ periodicities together give an $8$-periodic pattern in each fixed coweight. This is an $(8,8)$-periodicity in homotopy. It can be made precise by considering the following homotopy element.

\begin{defn}
	Let $\beta\in \pi_{8,8}(\cpltt{\kokq})$ be the element detected by $\NC{3}{}v_1^2 h_1^2$ such that $\tau^4\beta=v_1^4.$
\end{defn}

\begin{prop}
\label{alphabetarelation}
	There are relations that 
	\begin{enumerate}
		\item $\rho^3\beta=\eta\alpha.$
		\item $\alpha^2=4\beta.$
	\end{enumerate}
\end{prop}

\begin{proof}
We refer the reader to the coweight $0$ chart on Page \pageref{fig:sample}.

	\begin{enumerate}
		\item By the definition of $v_1^4$, we have $\tau^4\beta\cdot \rho^4=v_1^4\cdot \rho^4=\tau^4\eta^4.$ Therefore 
		$\beta\cdot \rho^4=\eta^4=\eta\cdot\rho\cdot\alpha.$ Since there is no $\rho$-torsion in degree $(5,5)$, we have that $\rho^3\beta=\eta \alpha.$
		\item By Proposition \ref{multiinHZ}, the element $\alpha^2$ is detected by 
	$$\left(\NCt{} v_1^2\right)^2=2\NCt{3}v_1^4,$$ 
	which is in filtration $4$.
Meanwhile, the element $4v_1^4=4\tau^4\beta$ is detected by $4 v_1^4=\tau^4\cdot 2\NCt{3}v_1^4$. 
Therefore, $4\beta$ is also detected by $2\NCt{3}v_1^4$. Since $\alpha^2$ and $4\beta $ are detected by the same element in filtration $4$, their difference $\Delta:=4\beta-\alpha^2$ with degree $(8,8)$ is detected by some element in filtration greater than or equal to $5$.

Meanwhile, by (1), we have that 
$$4\rho^3\beta=4\eta\alpha=\rho^2\eta^3\alpha=\rho^3\alpha^2.$$
Therefore $\rho^3\cdot \Delta=0$. Since there is no $\rho$ torsion in degree $(8,8)$ in filtrations $\geq 5$, the difference $\Delta$ equals $0$ and the relation holds.
	\end{enumerate}
\end{proof}

The discussion above shows the following theorem.
\begin{thm}
	Multiplication by $\beta$ gives an isomorphism
	$$\pi_{s,w}(\cpltt{\kokq})\to \pi_{s+8,w+8}(\cpltt{\kokq})$$
	 on homotopy groups.
\end{thm}

\begin{rmk}
	The element $\beta$ forgets to the classical Bott periodicity class. 
\end{rmk}

\subsection{Hidden \texorpdfstring{$\omega$}{omega} extensions}\hfill\\

Consider the element $2$ in degree $(0,0,0)$ in the $C_2$-effective $E_\infty$-page.  Because of the presence of the elements $(\rho h_1)^n$ in higher filtration, this element detects infinitely
many elements in $\pi_{0,0}^{C_2}(\cpltt{\kokq} )$. We also typically use the same symbol $2$ to represent twice the identity in $\pi_{0,0}(\cpltt{\kokq} )$.  The two meanings of this symbol are dangerously ambiguous, and we caution the reader to be careful about the distinction. 
%other elements in higher filtration, there is a previledged representitve in homotopy which is not twice the identity. over stating. dangerously ambiguous.

\begin{definition}
	Let $\omega$ be the element $\rho\eta+2\in \pi_{0,0}^{C_2}(\cpltt{\kokq} )$. 
\end{definition}

%Let $\omega$ be the element $2 + \eta \rho$ in $\pi_{0,0}$.  
The element $\omega$ is also detected by $2$ in the $C_2$-effective $E_\infty$-page.  In many ways, it has better homotopical properties than the element $2$ in homotopy. Note that $\omega$ equals $1-\epsilon$ where $\epsilon$ is the twist: $S^{1,1}\wedge S^{1,1}\to S^{1,1}\wedge S^{1,1}$ (\cite{dugger2013motivic}).
In particularly, it plays the role of the zeroth $C_2$-equivariant (and $\bbR$-motivic) Hopf map. One way to understand this is to consider the $C_2$-equivariant Adams spectral sequence, where the element $\omega$ is detected by $h_0$ and $2$ is detected by $h_0+\rho h_1$ (see \cite{GHIR}). 

\begin{comment}
to consider Steenrod operations on the cohomologies
of the cofiber of $\omega$ and the cofiber of $2$.  In the cofiber of $\omega$, there is a
$\Sq^1$ connecting the bottom and top cells, and no other operations except for those arising
within the cohomology of a point.  On the other hand, in the cofiber of $2$, there is both a 
$\Sq^1$ and a $\Sq^2$ connecting the bottom and top cells (see \cite{dugger2013motivic}).
\end{comment}
For this reason, we shall consider hidden $\omega$ extensions in $\pi_{*,*}$.  Information about multiplication by $2$ can then be easily obtained from the relation $2 = \omega -  \rho\eta$.

\begin{comment}
	Actually, the notation $2$ is ambiguous. In the $E_\infty$-page of the spectral sequence, it refer to twice the generator at degree $(0,0,0)$. It also refers to twice the identity in homotopy. The elements in the $E_\infty$-page represents a quotient set of homotopy elements. Because of the existence of higher filtration elements, the quotient set represented by $2$ is $2+\bbF_2\{\rho h_1\}$. In particular, it contains $2$ and $2+\rho h_1$. Both elements are detected by $2$. In fact, in $C_2$-equivariant(and $\bbR$-motivic) homotopy theory, the element $2+\rho h_1$ rather than $2$ plays the role of the zeroth Hopf map. We introduce the following notation.
\end{comment}

\begin{prop}
\label{omegarelation}
 There are relations:
 \begin{enumerate}
  \item $\omega\eta=0.$
  \item $\omega\rho=0.$
\end{enumerate}
	
\end{prop}
\begin{proof}
The element $\omega$ also lives in the $\bbR$-motivic stable homotopy groups of the sphere. By \cite{morel2004motivic} and \cite{dugger2013motivic}, there are $\bbR$-motivic relations $\omega\rho=0$ and $\omega\eta=0$. The same relations hold in the $C_2$-equivariant stable homotopy groups of the sphere by taking the $C_2$-equivariant Betti realization, and in the homotopy of $\kokq$ by applying the unit map.
\end{proof}

Using the relations in Proposition \ref{omegarelation}, we can obtain hidden $\omega$ extensions.

\begin{prop}
There are hidden $\omega$ extensions
\begin{enumerate}
	\item  from $\tau h_1$ to $\rho\tau h_1^2$. 
	\item  from $\NC{}{}h_1$ to $\NCt{} h_1^2$. 
\end{enumerate}

\end{prop}

\begin{pf}\mbox{}

\begin{enumerate}
  \item 
%  There is an $\bbR$-motivic hidden $\omega$ \begin{comment}By definition of $\omega$, it is sufficient to show the relation $2\cdot \tau\eta = 0$ in homotopy. 
  By \cite[Section 8.3]{dugger2017motivicadamssphere}, there is an $\bbR$-motivic relation $\omega \tau\eta=\rho \tau \eta^2$. This relation holds in the homotopy of the $C_2$-equivariant sphere by taking Betti realization, and in the homotopy of $\kokq$ by applying the unit map.
%	\end{comment}
	\item 
	This follows from $(1)$, since 
	$\NC{}{}h_1=\NC{}{2}\cdot \tau h_1$ and $\NCt{} h_1^2=\NC{}{2} \cdot \rho \tau h_1^2.$
%	It is sufficient to show that in homotopy, we have the relation that
%	$$2\cdot \NC{}{}\eta = 0.$$ 
%	The relation holds by the fact that $\NC{}{}\eta$ is decomposable as $\NC{}{2}\cdot \tau\eta$ and $2\cdot \tau\eta=0.$ 
\end{enumerate}

\end{pf}

\begin{rmk}
 The above two $\omega$ extensions generate other $\omega$ extensions under the $\tau^4$ and $v_1^4$ periodicities.	
\end{rmk}

\begin{prop}
	There is a hidden $\omega$ extension from $\NC{3}{}h_1 v_1^2$ to $\NCt{3}v_1^4.$	
\end{prop}

\begin{pf}

The element $\beta$ is detected by $\NC{3}{}h_1 v_1^2$. We will show that $\omega \beta$ is detected by $\NCt{3}v_1^4.$

By Proposition \ref{alphabetarelation}, the element $4\beta=\alpha^2$ is detected by $2\NCt{3}v_1^4.$ 
Therefore, the element $\omega \beta$ is detected in degree $(8,4,8)$ in the $E_\infty$-page, by $\NC{3}{}h_1v_1^2\cdot \rho h _1$, by $\NCt{3}v_1^4$, or by the sum of two. The homotopy elements that are detected by $\NC{3}{}h_1v_1^2\cdot \rho h _1$ or the sum $\NC{3}{}h_1v_1^2\cdot \rho h _1+\NCt{3}v_1^4$ are all $\rho$-torsion free. Since $\rho\cdot \omega\beta=0$ by Proposition \ref{omegarelation}, the only remaining possibility is that $\NCt{3}v_1^4$ detects $\omega\beta$.

\end{pf}

\section{Charts}
\label{charts}

In this section, we give the charts of the $E_1$ and $E_\infty$-pages of the $C_2$-effective spectral sequence for $\kokq.$
We organize this tri-graded spectral sequence by the coweight $s-w$. 

\subsection{\texorpdfstring{$E_1$}{E1}-page charts}

The charts on pages \pageref{fig:E11}-\pageref{fig:E12} depict the $E_1$-page of the $C_2$-effective spectral sequence for $\kokq$. The details are in Section \ref{esss} and Section \ref{EM}. 

The $E_1$-page is a $\bbZ_2[\tau^2]$-module, and its general description is presented by even/odd coweights. To avoid overlap, the positive cone (PC) elements and the negative cone (NC) elements of the same coweight are separated and shown on two different charts.

\subsection{\texorpdfstring{$E_2$}{E2}-page charts}

The charts on pages \pageref{fig:E21}-\pageref{fig:E24} depict the $E_\infty$-page of the $C_2$-effective spectral sequence for $\kokq$ with hidden extensions. The calculation details are in Section \ref{diff} and Section \ref{ext}. 

The $E_2$-page is a $\bbZ_2[\tau^4]$-module, and its general description is presented by mod $4$ coweights. Similar to the $E_1$ charts, the positive cone (PC) elements and the negative cone (NC) elements of the same coweight are shown on different charts.

When the coweight is $3 \pmod 4$ and $\geq -1$, the $E_\infty$-page is everywhere zero. Therefore, we include coweight $3 \pmod 4$ charts only for coweight $\leq -5.$ All elements in coweight $3 \pmod 4$ and $\leq -5$ are negative cone elements.

When the coweight is $0 \pmod 4$, some negative cone elements are $\tau^{-4}$ torsion. In the second chart on Page \pageref{fig:E24}, each dot is labelled with a number that indicates the order of the $\tau^{-4}$ torsion. The elements represented by the squares are $\tau^{-4}$ torsion free.

\subsection{Sample coweight 0 charts}

When coweight is $0 \pmod 4$, it is difficult to depict the results in full generality because of the interaction between positive and negative cone elements. To give a sense of the computation, we attach two sample charts which depict the $E_\infty$-page in coweight $0$ and coweight $4$ on page \pageref{fig:sample}.

\subsection{The key of the charts}\hfill

\label{key}

The key of the charts is as follows:

$~~(1)$ Dots and squares indicate groups as displayed on the charts.

$~~(2)$ Horizontal lines indicate multiplications by $\rho.$

$~~(3)$ Diagonal lines indicate multiplications by $h_1$.

$~~(4)$ Horizontal arrows indicate infinite sequences of multiplications by $\rho$.

$~~(5)$ Diagonal arrows indicate infinite sequences of multiplications by $h_1$.

$~~(6)$ Red dashed lines indicate hidden extensions by $\rho$.

$~~(7)$ Black dashed lines indicate hidden extensions by $h_1$.

$~~(7)$ Orange dashed lines indicate hidden extensions by $\omega$.

%\subsection{the \texorpdfstring{$E_\infty$}{E2} page}\hfill

%\label{chartsE2}
%Figure \ref{fig:E21}, Figure \ref{fig:E22}, Figure \ref{fig:E23}, Figure \ref{fig:E24} give a general description of the $E_\infty$-page. 
%
%The keys of the charts are as follows:\\
%$~~\bullet$ For the positive cone(PC):
%
%dots: a copy of $\bbF_2[\tau^4]$;
%
%squares: a copy of $\bbZ_2[\tau^4]$;\\
%$~~\bullet$ 	For the negative cone(NC):
%
%dots: a copy of $\bbF_2[\tau^{-4}]$;
%
%squares: a copy of $\bbZ_2[\tau^{-4}]$;\\
%$~~\bullet$ blue or red lines: multiplication by $\rho$;\\
%$~~\bullet$ dotted grey lines: hidden extension by $\eta$;\\
%$~~\bullet$ dotted blue lines: hidden extension by $\rho$;\\
%$~~\bullet$ dotted black lines: hidden extension by $\omega$;\\

\newpage

\subsection{The charts}\hfill
\label{subsec:charts}
%
%The red part is negative cone and the blue part is the positive cone. A square represents $\bbZ_2[\tau^2]$ for the positive cone and $\bbZ_2[\tau^{-2}]$ for the negative cone. A dot represents $\bbF_2[\tau^2]$ for the positive cone and $\bbF_2[\tau^{-2}]$ for the negative cons. Each black line segment of slope 1 represents multiplication by $h_1$ and a horizontal line segment of represents multiplication of $\rho.$
%\eject \pdfpagewidth=6.9in \pdfpageheight=11.2in
\begin{figure}[H]
\begin{center}
\makebox[0.95\textwidth]{\includegraphics[trim={0cm, 0cm, 0cm, 0cm},clip,page=2,scale=0.8]{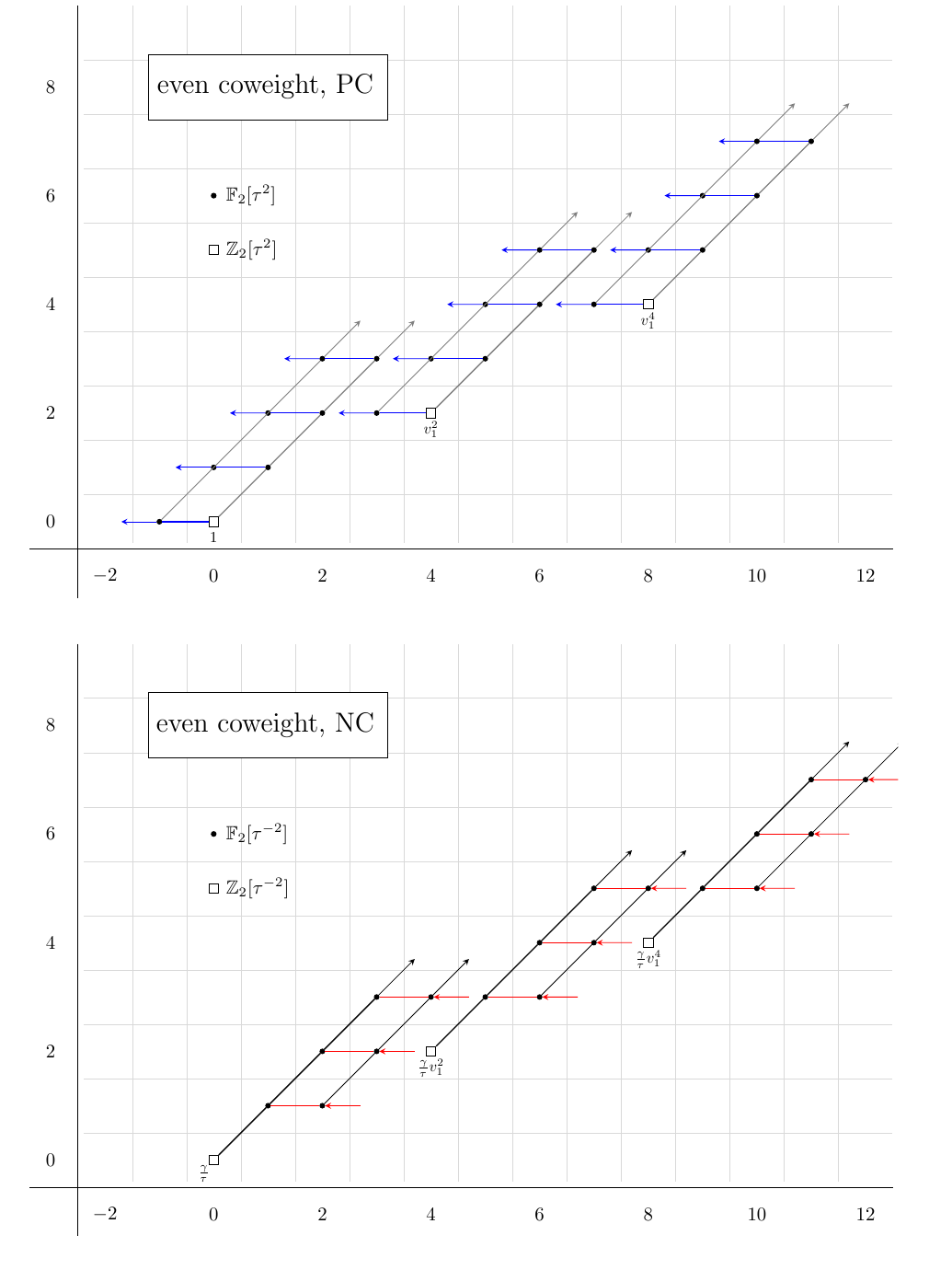}}
%\caption{The $E_1$-page, I}
\label{fig:E11}
\hfill
\end{center}
\end{figure}

\begin{figure}[H]
\begin{center}
\makebox[0.95\textwidth]{\includegraphics[trim={0cm, 0cm, 0cm, 0cm},clip,page=1,scale=0.8]{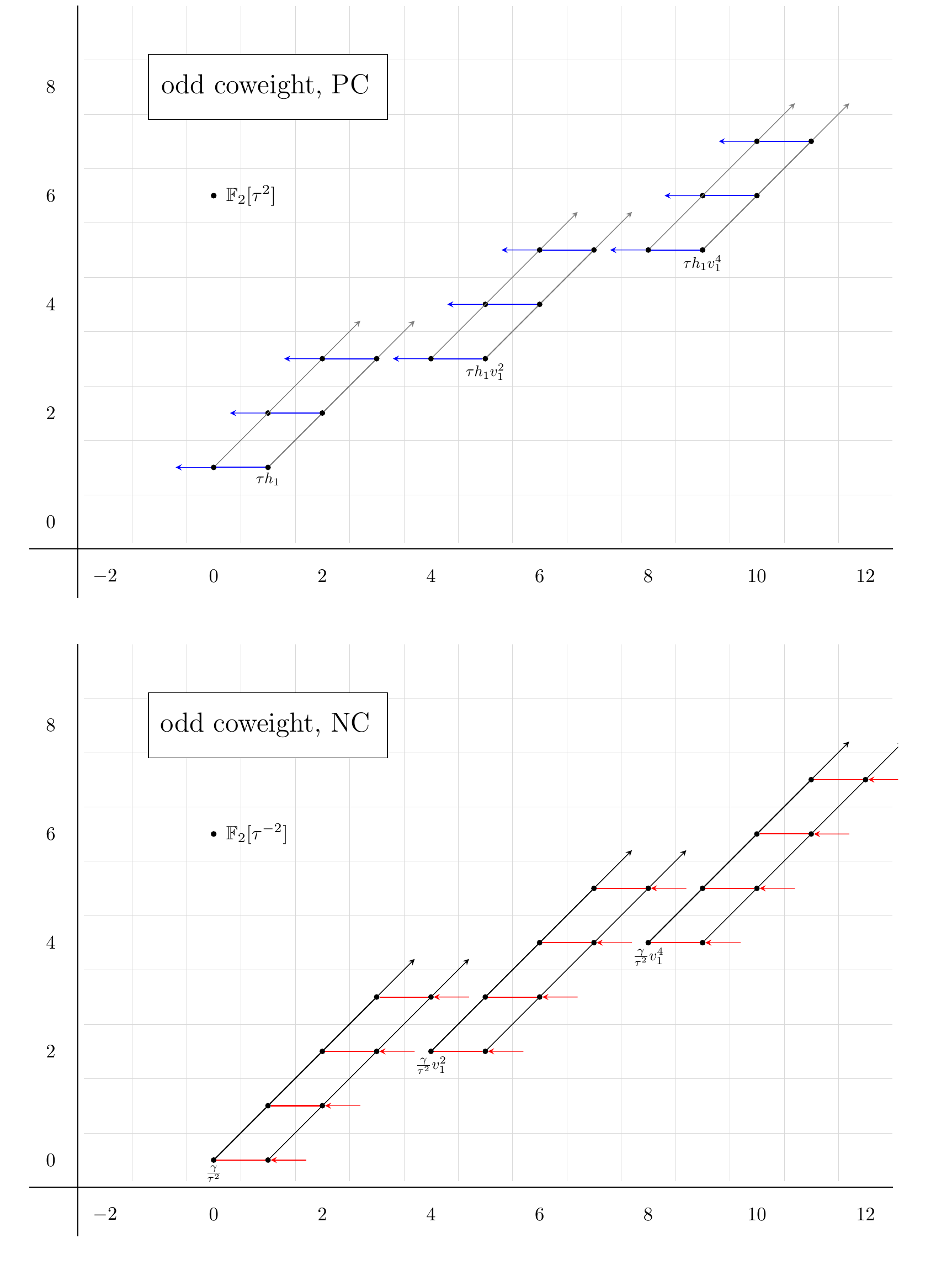}}
%\caption{The $E_1$-page, II}
\label{fig:E12}
\hfill
\end{center}
\end{figure}
\newpage

\eject \pdfpagewidth=15in \pdfpageheight=16in

\begin{figure}[H]
\begin{center}
{\includegraphics[trim={0cm, 0cm, 0cm, 0cm},clip,page=2]{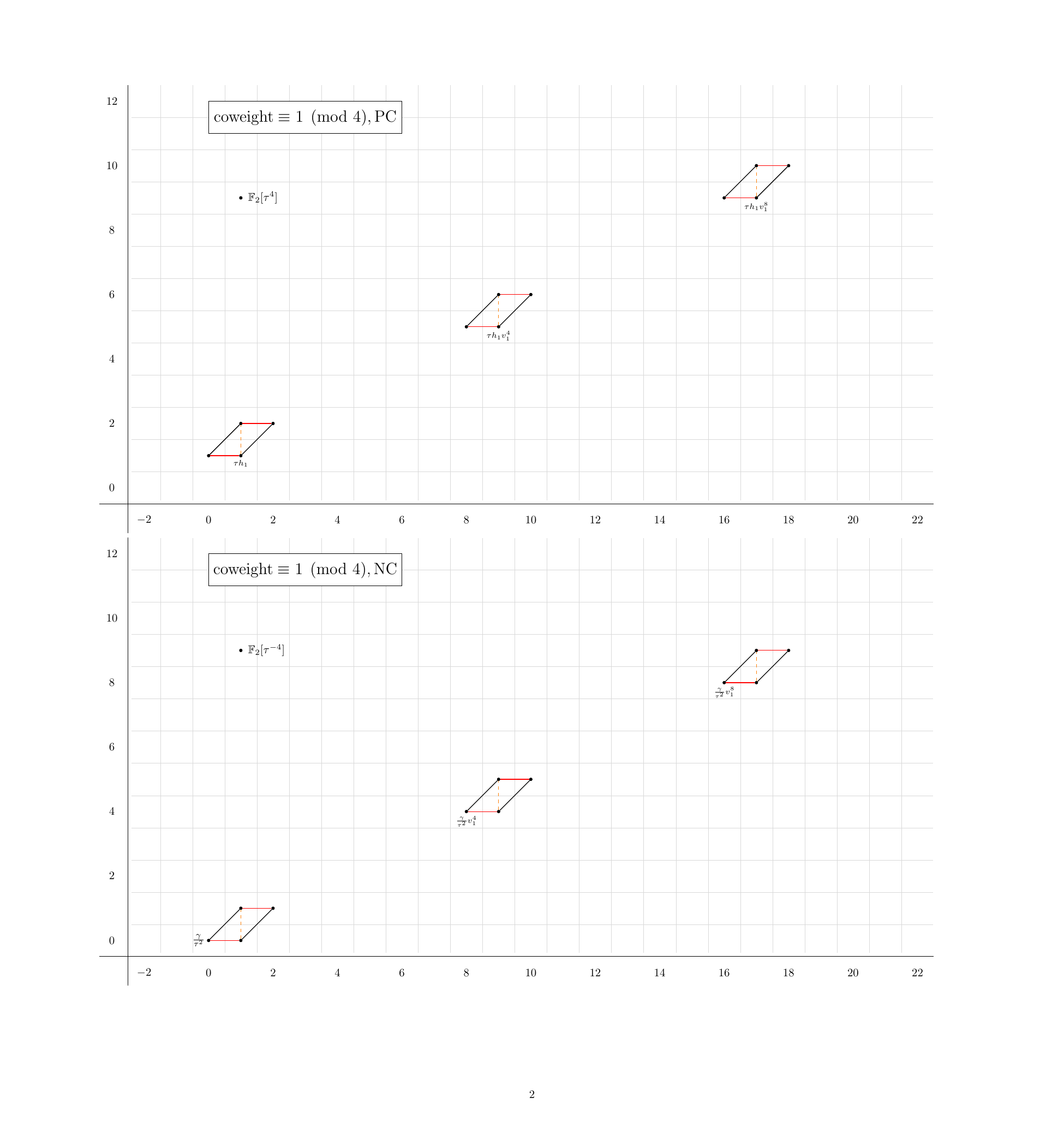}}
%\caption{The $E_\infty$-page, I}
\label{fig:E22}
\hfill
\end{center}
\end{figure}

\begin{figure}[H]
\begin{center}
{\includegraphics[clip,page=1]{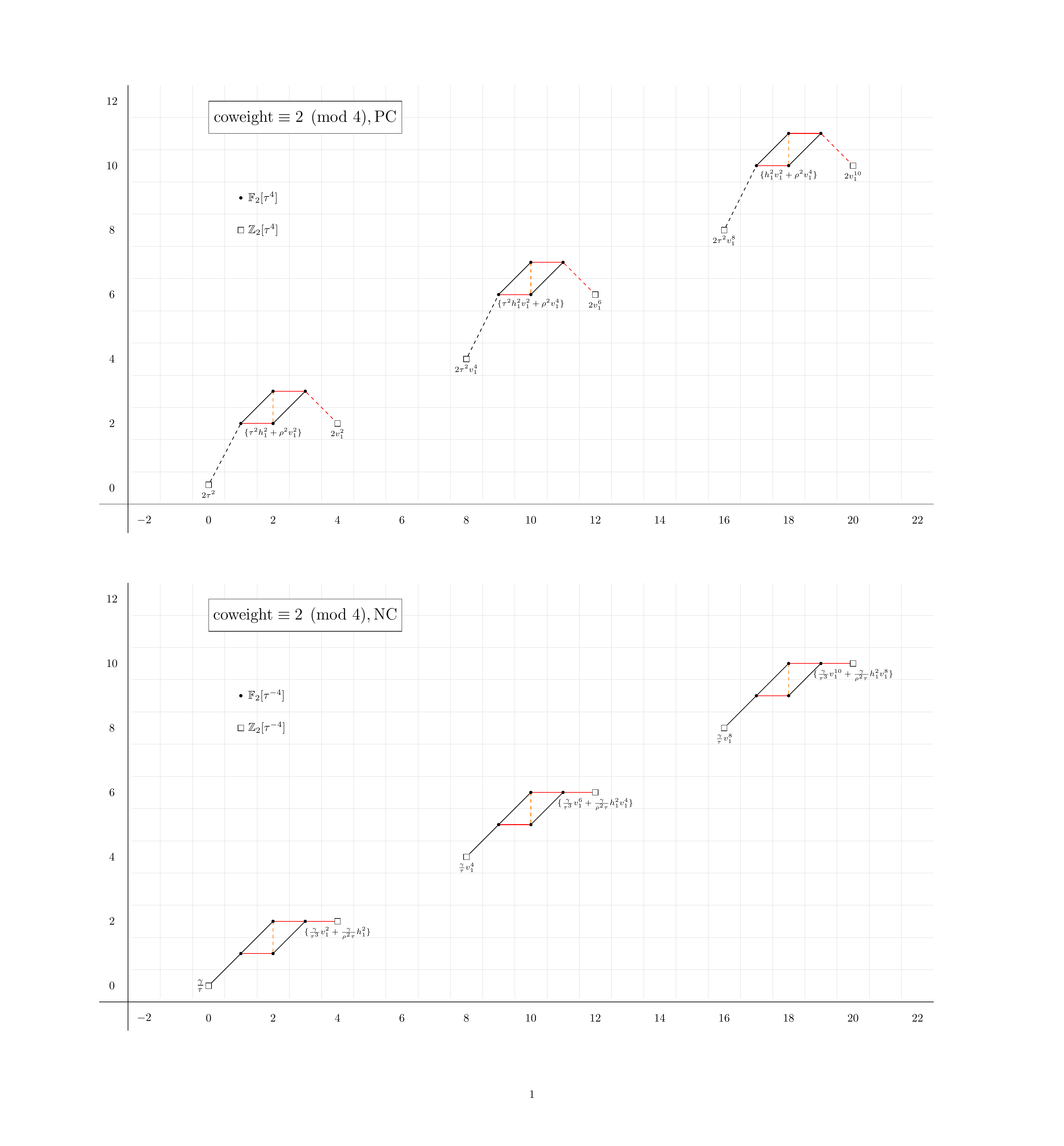}}
%\caption{The $E_\infty$-page, II}
\label{fig:E21}
\hfill
\end{center}
\end{figure}

\begin{figure}[H]
\begin{center}
{\includegraphics[trim={0cm, 0cm, 0cm, 0cm},clip,page=3]{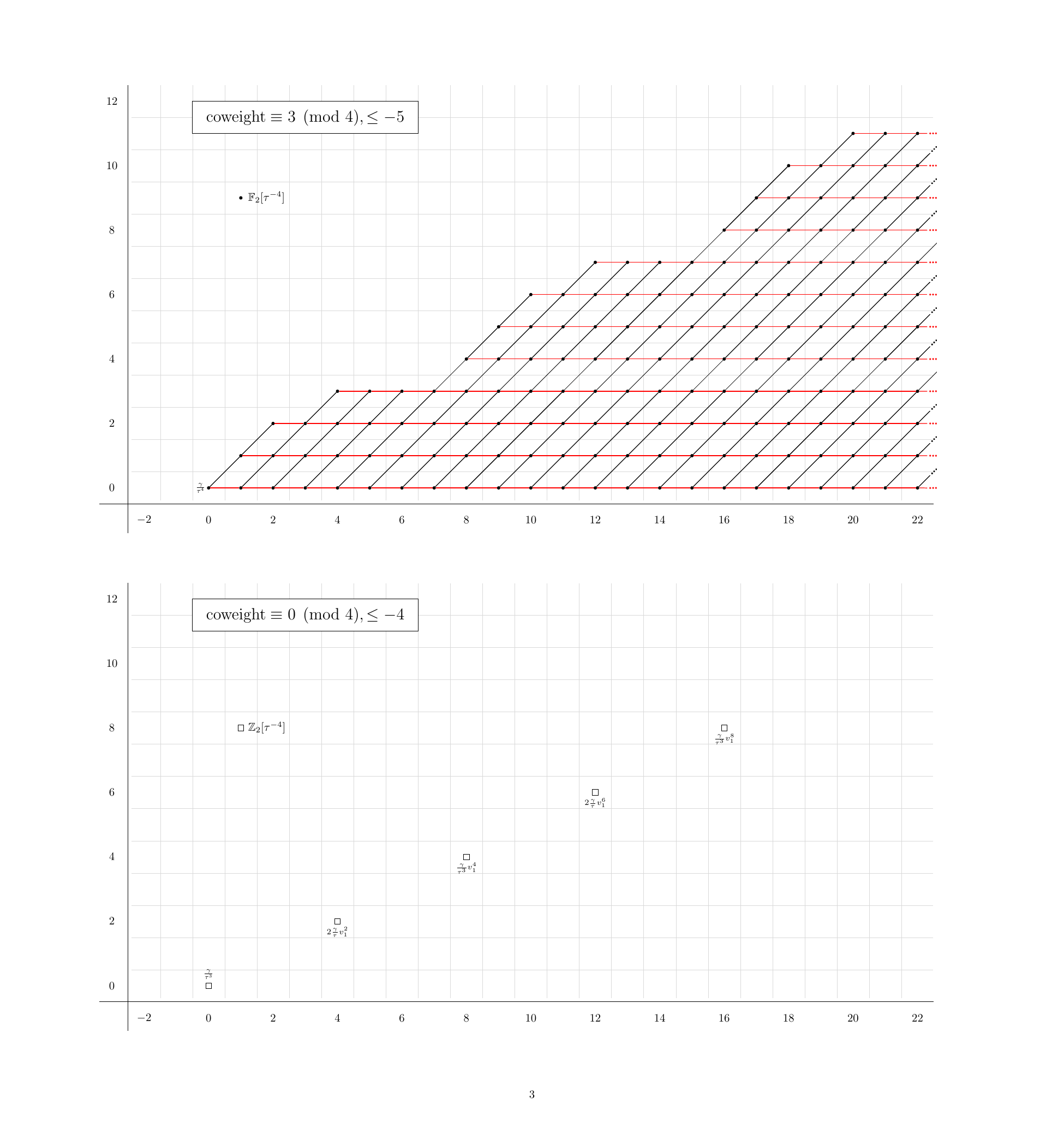}}
%\caption{The $E_\infty$-page, III}
\label{fig:E23}
\hfill
\end{center}
\end{figure}

\begin{figure}[H]
\begin{center}
{\includegraphics[trim={0cm, 0cm, 0cm, 0cm},clip,page=4]{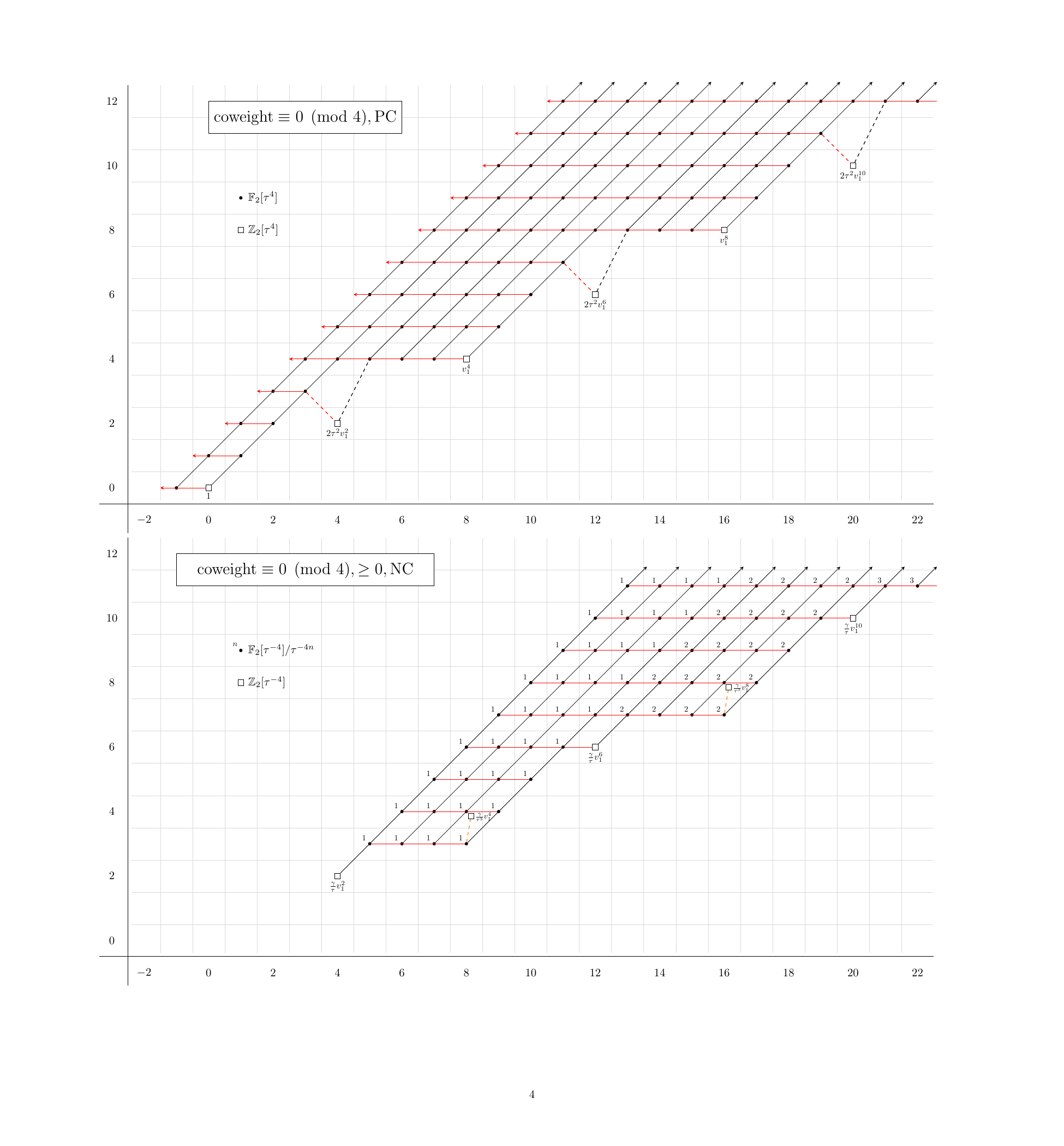}}
%\caption{The $E_\infty$-page, IV}
\label{fig:E24}
\hfill
\end{center}
\end{figure}

\begin{figure}[H]
\begin{center}
{\includegraphics[trim={0cm, 0cm, 0cm, 0cm},clip,page=5]{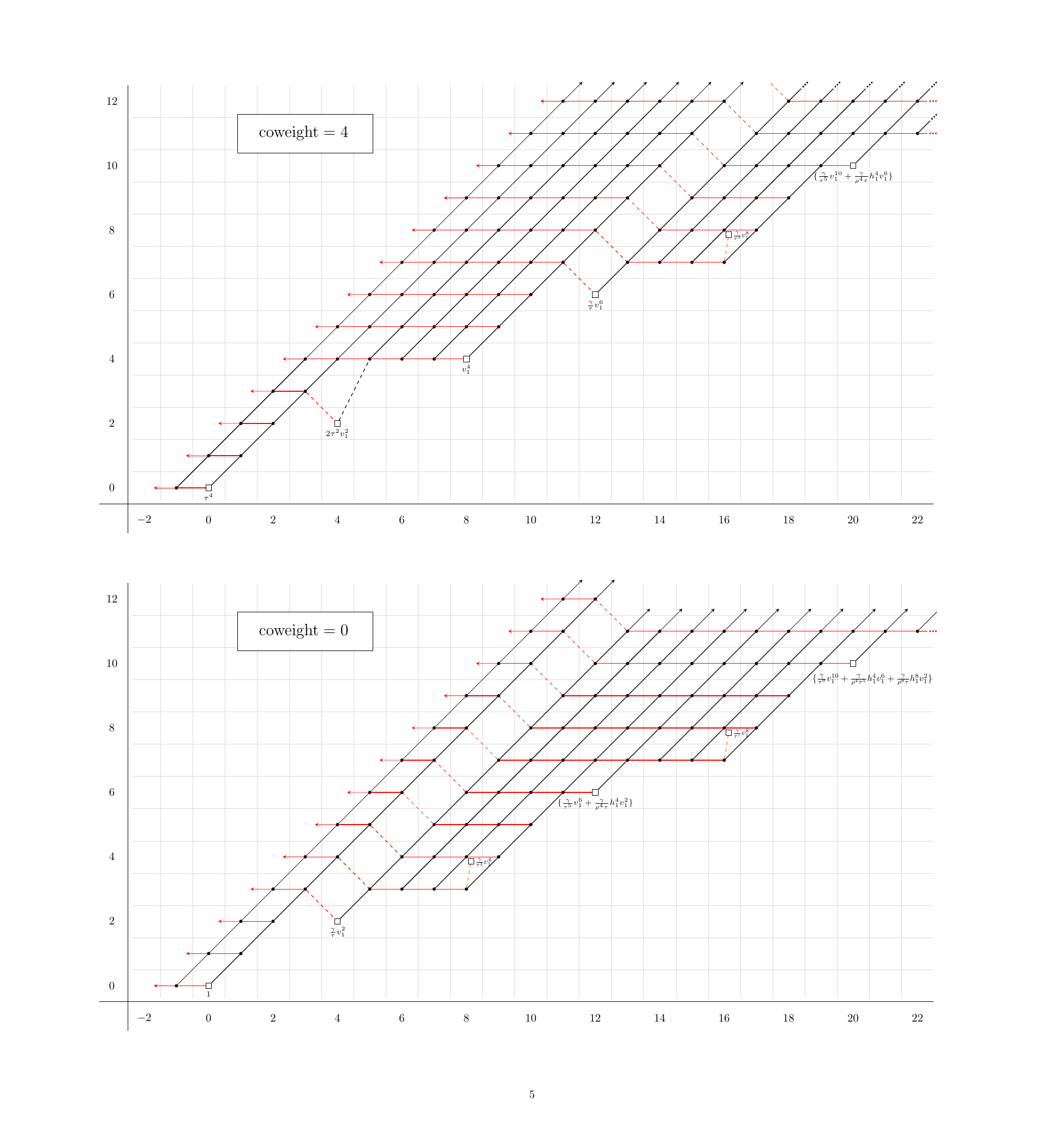}}
%\caption{The $E_\infty$-page, coweight=0, 4}
\label{fig:sample}
\hfill
\end{center}
\end{figure}

\newpage
\eject \pdfpagewidth=8.5in \pdfpageheight=11in

%%%%%%%%%%%%%%%%%%%%%%%%%%%%%%%%

%\begin{figure}
%\label{fig:E1}
%\begin{center}
%\makebox[0.45\textwidth]{\includegraphics[trim={0cm, 0cm, 0cm, 0cm},clip,page=1,scale=0.3]{charts/cw12per}}
%\hspace{20pt}
%\makebox[0.45\textwidth]{\includegraphics[trim={0cm, 0cm, 0cm, 0cm},clip,page=2,scale=0.3]{charts/cw12per}}\\
%\makebox[0.45\textwidth]{\includegraphics[trim={0cm, 0cm, 0cm, 0cm},clip,page=3,scale=0.3]{charts/cw12per}}
%\hspace{20pt}
%\makebox[0.45\textwidth]{\includegraphics[trim={0cm, 0cm, 0cm, 0cm},clip,page=4,scale=0.3]{charts/cw12per}}
%\caption{$E_1$-page by coweight}
%\hfill
%\end{center}
%\end{figure}

%\begin{figure}
%\label{fig:E2}
%\begin{center}
%\makebox[0.45\textwidth]{\includegraphics[trim={0cm, 0cm, 0cm, 0cm},clip,page=1,scale=0.3]{charts/cw22per}}
%\hspace{30pt}
%\makebox[0.45\textwidth]{\includegraphics[trim={0cm, 0cm, 0cm, 0cm},clip,page=2,scale=0.3]{charts/cw22per}}\\
%\makebox[0.45\textwidth]{\includegraphics[trim={0cm, 0cm, 0cm, 0cm},clip,page=3,scale=0.3]{charts/cw22per}}
%\hspace{30pt}
%\makebox[0.45\textwidth]{\includegraphics[trim={0cm, 0cm, 0cm, 0cm},clip,page=4,scale=0.3]{charts/cw22per}}
%\caption{$E_2$-page by coweight}
%\hfill
%\end{center}
%\end{figure}\newpage
%\begin{figure}
%\label{fig:E22}
%\begin{center}
%\makebox[0.45\textwidth]{\includegraphics[trim={0cm, 0cm, 0cm, 0cm},clip,page=5,scale=0.3]{charts/cw22per}}
%\hspace{30pt}
%\makebox[0.45\textwidth]{\includegraphics[trim={0cm, 0cm, 0cm, 0cm},clip,page=6,scale=0.3]{charts/cw22per}}
%\caption{$E_2$-page by coweight}
%\hfill
%\end{center}
%\end{figure}

\newpage

\bibliographystyle{alpha}
  \bibliography{bib.bib}

%\printbibliography
\end{document}